\documentclass[12pt,reqno]{amsart}
\pdfoutput=1

\usepackage[T1]{fontenc}

\usepackage{amssymb,amsmath,amsthm,mathtools,tabularx,bbm,url,microtype}
\usepackage{float,graphicx} 
\usepackage{mathabx}\changenotsign        
\usepackage{dsfont}
\usepackage[shortlabels]{enumitem}
\usepackage[utf8]{inputenc}
\allowdisplaybreaks

\usepackage[dvipsnames]{xcolor} 

\usepackage[pagebackref=true]{hyperref} 
\hypersetup{
	colorlinks,
	linkcolor={blue!60!black},
	citecolor={green!60!black},
	urlcolor={blue!60!black}
}

\usepackage{geometry}
\makeatletter
\def\moverlay{\mathpalette\mov@rlay}
\def\mov@rlay#1#2{\leavevmode\vtop{   \baselineskip\z@skip \lineskiplimit-\maxdimen
		\ialign{\hfil$\m@th#1##$\hfil\cr#2\crcr}}}
\newcommand{\charfusion}[3][\mathord]{
	#1{\ifx#1\mathop\vphantom{#2}\fi
		\mathpalette\mov@rlay{#2\cr#3}
	}
	\ifx#1\mathop\expandafter\displaylimits\fi}
\makeatother

\usepackage{tikz}
\usepackage{pgfplots}
\pgfplotsset{compat=1.18}
\usepackage{mathrsfs}
\usetikzlibrary{arrows}

\usepackage{tkz-euclide}
\tkzSetUpPoint[fill=black, size = 3pt]
\tkzSetUpLine[color=black, line width=0.6pt]

\definecolor{mygreen}{RGB}{106, 168, 79}
\definecolor{myblue}{RGB}{98, 160, 234}
\definecolor{myred}{RGB}{246,97,81}
\definecolor{black}{rgb}{0,0,0}

\usepackage{caption}
\usepackage{subcaption}

\usepackage[mathlines]{lineno}
\usepackage{etoolbox} 

\newcommand*\linenomathpatch[1]{%
	\expandafter\pretocmd\csname #1\endcsname {\linenomath}{}{}%
	\expandafter\pretocmd\csname #1*\endcsname{\linenomath}{}{}%
	\expandafter\apptocmd\csname end#1\endcsname {\endlinenomath}{}{}%
	\expandafter\apptocmd\csname end#1*\endcsname{\endlinenomath}{}{}%
}
\newcommand*\linenomathpatchAMS[1]{%
	\expandafter\pretocmd\csname #1\endcsname {\linenomathAMS}{}{}%
	\expandafter\pretocmd\csname #1*\endcsname{\linenomathAMS}{}{}%
	\expandafter\apptocmd\csname end#1\endcsname {\endlinenomath}{}{}%
	\expandafter\apptocmd\csname end#1*\endcsname{\endlinenomath}{}{}%
}

\expandafter\ifx\linenomath\linenomathWithnumbers
\let\linenomathAMS\linenomathWithnumbers
\patchcmd\linenomathAMS{\advance\postdisplaypenalty\linenopenalty}{}{}{}
\else
\let\linenomathAMS\linenomathNonumbers
\fi

\linenomathpatchAMS{gather}
\linenomathpatchAMS{multline}
\linenomathpatchAMS{align}
\linenomathpatchAMS{alignat}
\linenomathpatchAMS{flalign}
\linenomathpatch{equation}

\usepackage{cleveref}[capitalise, compress, nameinlink, noabbrev]

\theoremstyle{plain}
\newtheorem{theorem}{Theorem}[section]
\crefname{theorem}{Theorem}{Theorems}

\crefname{proposition}{Proposition}{Propositions}

\newtheorem{corollary}[theorem]{Corollary}
\crefname{corollary}{Corollary}{Corollaries}

\newtheorem{lemma}[theorem]{Lemma}
\crefname{lemma}{Lemma}{Lemmata}

\newtheorem{conjecture}[theorem]{Conjecture}
\crefname{conjecture}{Conjecture}{Conjectures}

\crefname{problem}{Problem}{Problems}

\newtheorem{claim}[theorem]{Claim}
\crefname{claim}{Claim}{Claims}

\crefname{observation}{Observation}{Observations}

\crefname{setup}{Setup}{Setups}

\newtheorem{fact}[theorem]{Fact}
\crefname{fact}{Fact}{Facts}

\crefname{algorithm}{Algorithm}{Algorithms}

\newtheorem{remark}[theorem]{Remark}
\crefname{remark}{Remark}{Remarks}

\crefname{example}{Example}{Examples}

\theoremstyle{definition}
\newtheorem{definition}[theorem]{Definition}
\crefname{definition}{Definition}{Definitions}

\crefname{construction}{Construction}{Constructions}

\crefname{question}{Question}{Questions}

\numberwithin{equation}{section}

\crefname{section}{Section}{Sections}
\crefname{appendix}{Appendix}{Appendix}

\crefname{figure}{Figure}{Figures}

\newcommand{\fw}[1]{\textcolor{black}{#1}}

\newcommand{\Prob}{\mathbb{P}}
\newcommand{\Exp}{\mathbb{E}}

\DeclarePairedDelimiter{\abs}{\lvert}{\rvert}

\DeclarePairedDelimiter{\floor}{\lfloor}{\rfloor}
\DeclarePairedDelimiter{\set}{\{}{\}}

\renewcommand{\epsilon}{\varepsilon}
\renewcommand{\rho}{\varrho}
\renewcommand{\ge}{\geqslant}
\renewcommand{\le}{\leqslant}
\renewcommand{\geq}{\geqslant}
\renewcommand{\leq}{\leqslant}
\renewcommand{\emptyset}{\varnothing}
\renewcommand{\setminus}{\smallsetminus}
\NewCommandCopy{\subs}{\subset}
\renewcommand{\subset}{\subseteq}

\newcommand{\defn}[1]{\textcolor{Maroon}{\emph{#1}}}
\newcommand{\eps}{\varepsilon}

\newcommand{\comin}{\delta^{\ast}}
\newcommand{\bR}{\mathbb{R}}
\newcommand{\bS}{\mathbb{S}}
\newcommand{\PG}[3]{{P^{(#3)}}(#1;#2)}

\newcommand\restr[2]{{
		\left.\kern-\nulldelimiterspace 
		#1 
		\vphantom{\big|} 
		\right|_{#2} 
}}

\newcommand{\cA}{\mathcal{A}}
\newcommand{\cB}{\mathcal{B}}

\newcommand{\cF}{\mathcal{F}}

\newcommand{\cK}{\mathcal{K}}
\newcommand{\cL}{\mathcal{L}}
\newcommand{\cM}{\mathcal{M}}

\newcommand{\cP}{\mathcal{P}}

\newcommand{\cR}{\mathcal{R}}

\newcommand{\cU}{\mathcal{U}}
\newcommand{\cV}{\mathcal{V}}

\title{Spanning spheres in Dirac hypergraphs}
\date{\today}

\author[F.~Illingworth]{Freddie Illingworth}
\author[R.~Lang]{Richard Lang}
\author[A.~M\"{u}yesser]{Alp M\"{u}yesser}
\author[O.~Parczyk]{Olaf Parczyk}
\author[A.~Sgueglia]{Amedeo Sgueglia}

\address[Illingworth]
{
	Department of Mathematics,
	University College London,
	London, UK
}
\email{f.illingworth@ucl.ac.uk}

\address[Lang]
{
	Departament de Matemàtiques,
	Universitat Politècnica de Catalunya,
	Barcelona, Spain
}
\email{richard.lang@upc.edu}

\address[M\"uyesser]
{
	New College,
	University of Oxford,
	Oxford, UK
}
\email{alp.muyesser@new.ox.ac.uk}

\address[Parczyk]
{
	Zuse Institute Berlin, 
	Department for AI in Society, Science, and Technology,
	Berlin, Germany
}
\email{parczyk@zib.de}

\address[Sgueglia]
{
	Fakult\"at f\"ur Informatik und Mathematik, 
	Universit\"at Passau, 
	Germany.
}
\email{amedeo.sgueglia@uni-passau.de}

\thanks{FI was supported by EPSRC grant EP/V521917/1 and the Heilbronn Institute for Mathematical Research.
	RL was supported by the EU Horizon 2020 programme MSCA (101018431).
	OP was funded by the DFG (EXC-2046/1, project ID: 390685689).
	While conducting this research, AS was affiliated with University College London and supported by the Royal Society.}

\begin{document}
	
	\begin{abstract}
		We show that a $k$-uniform hypergraph on $n$ vertices has a spanning subgraph homeomorphic to the $(k - 1)$-dimensional sphere provided that $H$ has no isolated vertices and each set of $k - 1$ vertices supported by an edge is contained in at least $n/2 + o(n)$ edges.
		This gives a topological extension of Dirac's theorem and asymptotically confirms a conjecture of Georgakopoulos, Haslegrave, Montgomery, and Narayanan.
		
		Unlike typical results in the area, our proof does not rely on the Absorption Method, the Regularity Lemma or the Blow-up Lemma.
		Instead, we use a recently introduced framework that is based on covering the vertex set of the host graph with a family of complete blow-ups.
	\end{abstract}
	
	\maketitle
	
	
	\section{Introduction}
	
	A seminal theorem of Dirac~\cite{Dirac} determines the best possible minimum degree condition for a graph to contain a Hamilton cycle.
	Over the past decades, a great number of generalisations of this result have been obtained~\cite{SS19}, often  in parallel with the development of new proof techniques such as the modern form of the Absorption Method~\cite{RRS08a} as well as the Blow-up Lemma~\cite{KSS98}.
	In many of these Dirac-type results, the embedded cyclical structures are organised along a linear ordering of the vertex set, an inherently $1$-dimensional concept.
	We study a different, topological direction of this problem with early traces in the work of Brown, Erdős, and Sós~\cite{SEB73} that concerns the existence of subgraphs with rich, high-dimensional structure in hypergraphs with good minimum degree conditions.
	
	To motivate this area, we first observe that a Hamilton cycle in a $2$-graph is a set of edges such that the simplicial complex induced by these edges is homeomorphic to the $1$-dimensional sphere, $\bS^1$, with the additional property that the $0$-skeleton of the complex comprises the entire vertex-set of the host graph.
	This topological viewpoint led Gowers and, independently, Conlon (see~\cite{georgakopoulos2022spanning}) to ask which degree conditions force a $3$-graph to contain a spanning copy of the $2$-dimensional sphere $\mathbb{S}^2$.
	To clarify, a \defn{copy of a $(k - 1)$-sphere} in a $k$-graph~$H$ is a set of edges that induce a homogeneous simplicial complex homeomorphic to the $(k - 1)$-sphere,~$\mathbb{S}^{k - 1}$, and such a copy is called \defn{spanning} if the $0$-skeleton of the simplicial complex is the entire vertex-set of $H$.
	
	The original question of Gowers and Conlon was stated in terms of codegrees.
	Formally, the \defn{minimum codegree} of a $k$-graph is the maximum integer $m$ such that every set of $k-1$ vertices is contained in at least $m$ edges.
	For $k=3$, Georgakopoulos, Haslegrave, Montgomery, and Narayanan~\cite{georgakopoulos2022spanning} gave an asymptotically optimal bound by proving that every $n$-vertex $3$-graph with minimum codegree at least $n/3 + o(n)$ contains a spanning copy of a $2$-sphere.
	It is an open conjecture whether this result can be extended to higher uniformities, see also \cref{sec:open-problems}.
	
	One of the main obstacles to combinatorial embeddings of spheres is topological connectivity.
	We capture this notion as follows.
	For a $k$-graph $H$, let $\hat H$ be the \defn{line graph} on the vertex set $E(H)$ with an edge $ef$ whenever $\abs{e \cap f} = k - 1$.
	A subgraph of $H$ is \defn{tightly connected} if it has no isolated vertices and its edges induce a connected subgraph in $\hat{H}$.
	Moreover, we refer to edge maximal tightly connected subgraphs as \defn{tight components}.
	It is easy to see that a $(k - 1)$-sphere is tightly connected.
	Thus, if a $k$-graph $H$ contains a vertex spanning sphere, then $H$ must have a spanning tight component.
	We can use this observation to obtain lower bounds for extremal problems.
	For instance, there are simple $3$-uniform constructions with codegree $n/3 - 1$ and no spanning tight component, and so the above result is sharp.
	
	In light of this, it is natural to ask which codegree condition that is locally supported by a tight component forces a spanning sphere.
	Formally, the \defn{minimum supported co-degree}, $\comin(H)$, of a non-empty $k$-graph $H$, is the maximum integer $m$ such that every $(k - 1)$-set contained in one edge of $H$ is contained in at least $m$ edges of $H$.
	A plausible first guess is that every tightly connected $3$-graph $H$ with $\delta^\ast(H)\geq n/3+o(n)$ contains a spanning copy of a $2$-sphere.
	However, constructions show that in fact $\delta^\ast(H)\geq n/2$ is necessary and this extends to all uniformities $k\geq 2$
	(see \cref{sec:constructions}).
	Given this, Georgakopoulos et al.~\cite{georgakopoulos2022spanning} posed the following conjecture.
	
	\begin{conjecture}\label{conj:mainconj}
		Every tightly connected $k$-graph $H$ of order $n> k \geq 2$ with $\comin(H) \ge n/2$ contains a spanning copy of $\bS^{k - 1}$.
	\end{conjecture}
	
	We remark that $2$-uniform case of the conjecture corresponds to Dirac's theorem.
	Our main result confirms \cref{conj:mainconj} asymptotically.
	
	\begin{theorem}\label{thm:maintheorem}
		For every $k \geq 2$ and $\eps>0$, there exists $n_0$ such that every tightly connected $k$-graph $H$ of order $n \geq n_0$ with $\comin(H) \ge (1/2+\eps)n$ contains a spanning copy of $\bS^{k - 1}$.
	\end{theorem}
	
	Spheres have been studied before in other combinatorial settings.
	Historically, the notion can be traced back to the work of Brown, Erdős, and Sós~\cite{SEB73} in the seventies.
	More recently, Luria and Tessler~\cite{LT19} studied spanning spheres in random complexes.
	In the graph setting, Kühn, Osthus, and Taraz~\cite{KOT05} gave asymptotically optimal minimum degree conditions for the existence of a $2$-dimensional sphere in graphs, later generalised by the Bandwidth theorem of B\"{o}ttcher, Schacht, and Taraz~\cite{BST09}.
	A sharper version of the former result was obtained by Kühn and Osthus~\cite{KO05}.
	Finally, the notion of minimum supported degrees was recently studied in the context of Tur\'an-type problems~\cite{Bal21,Pik23}.
	
	We remark that unlike in the graph setting it is not possible to derive optimal minimum supported degree conditions for tight spheres from a bandwidth theorem.
	This is because in the context of hypergraphs, a bandwidth theorem requires structure (namely a tight Hamilton path) with higher degree thresholds than those needed for spheres.
	See \cref{sec:open-problems} for a discussion of tight Hamilton cycles and paths.
	
	Nevertheless, our approach benefits from ideas that have been developed to tackle the problem of embedding hypergraphs of bounded bandwidth.
	In particular, \cref{lem:findblowupchain} is an instance of a forthcoming more general setup~\cite{LS24}, which has been tailored to the requirements of embedding spanning spheres.
	
	\subsection{Proof outline}
	
	There are two well-known approaches to Dirac-type problems in graphs and hypergraphs.
	The first approach consists of a combination of Szemer\'{e}di's Regularity Lemma and the Blow-up Lemma due to Koml\'{o}s, S\'{a}rk\"{o}zy, and Szemer\'{e}di~\cite{KSS98}.
	The other approach is based on the Absorption Method, whose modern form was introduced by R\"{o}dl, Ruci\'{n}ski, and Szemer\'{e}di~\cite{RRS08a}, often also in conjunction with a (hypergraph) Regularity Lemma.
	Georgakopoulos et al.~\cite{georgakopoulos2022spanning} followed the second approach when embedding spanning $2$-spheres in $3$-uniform hypergraphs using a $2$-uniform version of the Regularity Lemma.
	One of the main difficulties in extending their method to higher uniformities is that one needs to invoke a stronger ($k$-uniform) version of the Regularity Lemma.
	While this is presumably not prohibitive, it leads to a quite technical setup with many additional challenges (see for instance~\cite{LS22}).
	In order to avoid such complications, our proof is based on an upcoming framework for embedding large structures into hypergraphs of Lang and Sanhueza-Matamala~\cite{LS24}, which has a precursor in the setting of perfect tilings~\cite{lang2023tiling}.
	
	\subsubsection*{Classic approach}
	This new approach is conceptually inspired by the classic combination the Regularity Lemma and the Blow-up Lemma, which we briefly recapitulate in the following.
	Broadly speaking, the Regularity Lemma allows us to approximate a given host-graph $H$ with a quasirandom blow-up of a \emph{reduced graph} $R$ of constant order, meaning that vertices of $R$ are replaced with (disjoint, linear sized) vertex sets of $H$ and the edges of~$R$ are replaced by quasirandom partite graphs.
	Importantly, the Regularity Lemma guarantees that $R$ approximately inherits degree conditions.
	The Blow-up Lemma then tells us that one can effectively assume that these (quasirandom) partite graphs are complete.
	Ignoring a number of technicalities, this reduces the problem of embedding a large structure $S$ into $H$ to the problem of embedding a large structure $S$ into a complete blow-up of $R$.
	This last step is known as the \emph{allocation} of $S$, and there is now a broad set of techniques available to facilitate this step. Therefore, the reduction to the allocation problem often presents considerable step towards the successful embedding of $S$ into $H$.
	
	\subsubsection*{Our approach}
	In comparison to this, our approach proceeds as follows.
	Instead of approximating the structure of the host graph $H$ with a single quasirandom blow-up, we cover its vertex set $V(H)$ with a family of complete blow-ups $R_1^\ast,\dots,R_\ell^\ast$ contained in $H$, whose reduced graphs $R_1,\dots,R_\ell$ inherit the degree conditions of $H$.
	Importantly, the blow-ups are interlocked in a path-like fashion.
	So, in order to embed a given guest graph $S$ into $H$, we first find $S_1,\dots,S_\ell$ such that $\bigcup S_i = S$, and satisfying that $S_i$ and $S_j$ are fully disjoint unless $j=i+1$. Subsequently, we embed each $S_i$ into the blow-up of $R_i$ such that the connections between $S_i$ and $S_{i+1}$ are in accordance.
	This effectively reduces the problem of embedding $S$ into $H$ to embedding each $S_i$ into a complete blow-up of $R_i$, which is analogous to the above detailed allocation problem.
	
	\subsubsection*{Allocation}
	To find a sphere $S_i$ that covers the blow-up $R_i^\ast$ of $R_i$, we first construct a preliminary sphere $S_i'$ that has (for simplicity of the sketch) a suitable ``entry facet'' $F_e \in E(R_i^\ast)$ within the blow-up of every edge $e \in E(R_i)$.
	This is possible since $R_i$ is tightly connected.
	Notably, $S_i'$ can be taken to be very sparse, as the order of $R_i$ is much smaller than the blow-up part sizes.
	This sparsity allows us to ignore the perturbation created by reserving~$S_i'$.
	We then extend $S_i'$ by gluing additional spheres $B_e$ onto the faces $F_e$, where each $B_e$ is contained in the blow-up of the edges $e \in E(R_i)$.
	Importantly, the spheres $B_e$ are disjoint from each other and from the sphere $S_i'$ (apart from the designated entry facets).
	The main difficulty here is to compute the appropriate sizes of the spheres $B_e$, which we achieve by solving a $2$-dimensional matching problem using Dirac's theorem.
	Together, this leads to the desired spanning sphere $S_i$ in $R^\ast_i$.
	In order to connect the spheres $S_1,\dots,S_t$, we carefully place additional facets (at intersections between each $R_i^\ast$ and $R^\ast_{i+1}$) such that the spanning sphere $S$ can be obtained by gluing together the spheres $S_1,\dots,S_t$ along these facets.
	This is possible since the blow-ups $R_1^\ast, \dots, R_\ell^\ast$ are interlocked in a path-like fashion.

	\subsection{Open problems}\label{sec:open-problems}
	
	Our work leads to a series of interesting open problems, which we discuss in the following.
	
	\subsubsection*{Connectivity}
	
	Georgakopoulos et al.~\cite{georgakopoulos2022spanning} conjectured that any $n$-vertex $k$-graph with minimum codegree at least $n/k$ contains a spanning $(k - 1)$-sphere and confirmed the case $k=3$ approximately.
	We believe that our framework is suitable for attacking the general conjecture.
	A first step towards this would be to show that such a $k$-graph contains a spanning tight component, which is an interesting problem on its own.
	
	\begin{conjecture}
		Every $k$-graph on $n$ vertices with minimum codegree at least  $n/k$ contains a vertex spanning tight component.
	\end{conjecture}
	
	We remark that this minimum degree is best possible due to the constructions for the original conjecture~\cite{georgakopoulos2022spanning}.
	
	Another question in this direction raised in ~\cite{georgakopoulos2022spanning} is how large a sphere can be found in a $3$-graph with a given minimum codegree below $n/3$. As noted in~\cite{georgakopoulos2022spanning}, it is possible here that tight components are far from spanning, but in this situation pairs supported within small tight components will have large codegree relative to the size of the component, and hence our \cref{thm:maintheorem} could potentially help in resolving this problem.

	\subsubsection*{Tight cycles}
	
	Given our main result, it is natural to ask about the minimum supported codegree threshold for tight Hamilton cycles.
	Formally, a $k$-uniform \defn{tight cycle} comes with a cyclical ordering of its vertices such that every $k$ consecutive vertices form an edge.
	A classic result of R\"{o}dl, Ruci\'{n}ski, and Szemer\'{e}di~\cite{RRS08a} states that an $n$-vertex graph with minimum codegree at least $n/2 + o(n)$ contains a tight Hamilton cycle.
	As it turns out, this is no longer true for supported degree.
	
	Indeed, consider an $n$-vertex $k$-graph $G$ with $n$ divisible by $k$ and a partition $X \cup Y$ of its vertex-set with $\abs{Y} = n/k + 1$ such that $G$ contains all edges with at least $k - 1$ vertices in $X$.
	Note that any matching in $G$ misses at least one vertex of $Y$, so $G$ does not have a perfect matching and by extension no tight Hamilton cycle.
	On the other hand, $G$ satisfies $\delta^\ast(G) = (1 - 1/k)n - (k - 1)$.
	
	We believe that these constructions are optimal.
	
	\begin{conjecture}
		For every $k \geq 2$, there is some $n_0$ such that every $k$-uniform tightly connected $G$ on $n \geq n_0$ vertices with $\comin(G) \geq (1 - 1/k) n$ contains a tight Hamilton cycle.
	\end{conjecture}

    Very recently, this conjecture has been confirmed approximately by Mycroft and Zárate-Guerén~\cite{MZ25}.
	
	\subsubsection*{Exact bounds}
	Our main result gives an asymptotic solution to \cref{conj:mainconj}.
	To prove the exact conjecture, further work is required possibly including a stability analysis.
	It is conceivable that our embedding setup, namely \cref{lem:findblowupchain}, can be extended to account for this by using tools from the theory of property testing (see \cite{lang2023tiling} for more details).
	On the other hand, removing the error term in \cref{thm:maintheorem}, would also require a more precise analysis of the host graph structure, which could pose a significant challenge.
	We note that even in the setting of tight Hamilton cycles in $k$-graphs, which has been much more heavily investigated than the topological variant considered here, exact results are only known for $k \leq 3$~\cite{Dirac,rodl2011dirac}.
	
	\subsubsection*{Vertex degree}
	
	Since the $3$-uniform minimum codegree threshold for spanning $2$-spheres is well-understood, at least in the approximate sense, it is natural to ask what happens for $2$-spheres under minimum vertex-degrees.
	The following construction gives a lower bound.
	
	Consider a $3$-graph $G$ whose vertices are partitioned by sets $X$, $Y$, and $Z$ each of size $n/3$ and whose edges are composed of all edges of type $XXY$, $YYZ$, and $ZZX$ as well as all edges inside each of $X$, $Y$, and $Z$.
	It is not hard to see that $G$ does not have a spanning tight component and hence no spanning $2$-sphere.
	On the other hand, a simple calculation shows that every vertex of $G$ is on at least $(4/9 - o(1)) \binom{n}{2}$ edges.
	
	It is plausible, that $3$-graphs above this degree contain spanning $2$-spheres.
	To formalise this, we denote by $\delta_1(G)$ the maximum $m$ such that every vertex of a $3$-graph $G$ is contained in at least $m$ edges.
	
	\begin{conjecture}
		For every $k \geq 2$ and $\eps > 0$, there is $n_0$ such that any $k$-graph $G$ on $n\geq n_0$ vertices with $\delta_1(G) \geq (4/9 + \eps ) \binom{n}{2}$ contains a spanning $2$-sphere.
	\end{conjecture}
	
	Note that this conjecture implicitly asserts that the minimum (relative) vertex degree threshold for a spanning tight component is $4/9$, which is itself an interesting problem.
	On the other hand, it can be shown that the threshold for spanning spheres is at most $5/9$.
	This follows from a combination of the facts that $5/9$ is the threshold for tight Hamilton paths~\cite{RRR19}, certain blow-ups of tight paths containing spanning spheres (\cref{lem:thin-path}) and an upcoming hypergraph bandwidth theorem~\cite{LS24}.

	\subsection{Organisation of the paper}
	
	In \cref{sec:preliminaries} we collect notation and definitions which we will use throughout the paper.
	In the same section, we state our two key lemmata (\cref{lem:allocation,lem:findblowupchain}) and then prove our main result (\cref{thm:maintheorem}) assuming them.
	\cref{sec:tools} collects known tools, needed for our proofs.
	Finally, we prove \cref{lem:findblowupchain} in \cref{sec:blow-up} and \cref{lem:allocation} in \cref{sec:geometric_obs,sec:allocation}.
	
	\subsection*{Acknowledgements} We would like to thank an anonymous referee for a careful reading of the paper, and in particular for pointing us to an inaccuracy in the proof of Lemma~\ref{lem:findblowupchain} in an earlier version of this manuscript.

	\section{Preliminaries, key lemmata, and proof of the main result}
	\label{sec:preliminaries}
	
	\subsection{Notation and definitions}
	
	A \defn{$k$-uniform hypergraph} $H$ (or $k$-graph for short) consists of a set of \defn{vertices} $V(H)$ and a set of \defn{edges} $E(H)$, where each edge is a set of $k$ vertices.
	For a subset $S \subset V(H)$, we denote by  $\deg_H(S)$ the number of edges $e\in E(H)$ such that $S\subseteq e$.
	The \defn{minimum codegree} of $H$, denoted by $\delta(H)$, is the maximum integer $m$ such that every $(k - 1)$-set has degree at least $m$ in $H$.
	A set $S$ is \defn{supported} in $H$ if it has positive degree.
	We write \defn{$\partial H$} for the set of supported edges in $H$.
	
	A \defn{blow-up} of a $k$-graph $F$ is obtained by replacing each each vertex $x \in V(F)$ by a non-empty vertex set $V_x$ and each edge $e = x_1 \dotsc x_k \in E(F)$ by a complete $k$-partite $k$-graph on parts $V_{x_1}, V_{x_2}, \dotsc, V_{x_k}$. If $F^\ast$ is a blow-up of $F$, then there is a \defn{projection} map $\phi \colon V(F^\ast) \to V(F)$ defined by $\phi(v) = x$ for $v \in V_x$.
	Moreover, we write $\phi(U) = \set{\phi(v)\colon v \in U}$ for $U\subset V(F^\ast)$.

	We call $F^\ast$ a \defn{$(\gamma, m)$-regular blow-up} of $F$ if $F^\ast$ is a blow-up of $F$ where each part has size in the interval $[(1 - \gamma)m, (1 + \gamma)m]$.
	We abbreviate to \defn{$m$-regular blow-up} if $\gamma = 0$.
	Finally,~$F^\ast$ is called a \defn{$(\gamma,m)$-nearly-regular blow-up} if $F^\ast$ is a blow-up of $F$ where all but at most one part has size in the interval $[(1 - \gamma)m, (1 + \gamma)m]$, and the size of the (at most one) remaining exceptional part is exactly one. Note that every $(\gamma,m)$-regular blow-up is, in particular, a $(\gamma,m)$-nearly-regular blow-up.
	
	\subsection{Key lemmata}
	As anticipated in the introduction, our strategy is as follows: we would like to partition the graph into \emph{well-behaved} pieces, meaning that each piece can be covered with a sphere and that it is possible to glue all such spheres together into a spanning one.
	We now give more details and state the two key lemmata.
	
	Let $H$ be a $k$-graph with no isolated vertices and $\comin(H) \ge (1/2+o(1))n$.
	Methods from \cite{lang2023tiling} (see \cref{sec:property_graph}) can already cover all but at most $o(n)$ vertices of $H$ with a family $\cF$ of pairwise vertex-disjoint $k$-graphs with $\abs{\cF}=O(1)$ such that each $k$-graph $F^\ast \in \cF$ is the blow-up of a graph $F$ with no isolated vertices and $\comin(F) \ge (1/2+o(1))\abs{V(F)}$.
	However, this alone, is not sufficient for us.
	Indeed we want to cover the leftover as well and make sure that the parts are well-connected in a path-like fashion as described above.
	
	In order to achieve that, we modify the graphs in $\cF$ to get a sequence of $k$-graphs $F_1^\ast, \dotsc, F_{\ell}^\ast$ with $\ell=O(1)$ satisfying the following properties.
	Each $F_i^\ast$ is the nearly-regular blow-up of a $k$-graph $F_i$ with no isolated vertices and $\comin(F_i) \ge (1/2+o(1))\abs{V(F_i)}$; each vertex of $H$ is covered by at least one $F_i^\ast$; each $F_i^\ast$ can share vertices only with the $k$-graph coming before and the one coming after in the sequence.
	More precisely, $F_i^\ast$ and $F_{i+1}^\ast$ share exactly $k$ vertices which induce an edge in both $F_i^\ast$ and $F_{i+1}^\ast$.
	The formal statement of the first key lemma is as follows.
	\begin{lemma}[Blow-up chain]\label{lem:findblowupchain}
		Let $1/n\ll 1/m_2\ll 1/m_1 \ll 1/s, \gamma \ll \eps, 1/k\leq 1$, and let $H$ be an $n$-vertex $k$-graph with $\comin(H) \ge (1/2+\eps)n$ and no isolated vertices. Then, there exists a sequence of $s$-vertex $k$-graphs $F_1,\dotsc, F_\ell$ and a sequence of subgraphs $F_1^\ast,\dotsc, F_\ell^\ast\subseteq H$ such that the following properties hold for each $i \in [\ell]$ and $j \in [\ell-1]$:
		\begin{enumerate}[label = \textup{(}\arabic*\textup{)}]
			\item \label{blowup_1} $F_i$ has no isolated vertices and $\comin(F_i) \ge (1/2+\eps/2)\abs{V(F_i)}$,
			\item \label{blowup_2} there exists an $m_i^\ast\in [m_1,m_2]$ such that $F_i^\ast$ is a $(\gamma, m_i^\ast)$-nearly-regular blow-up of $F_i$,
			\item \label{blowup_3} $V(F_1^\ast) \cup \dots \cup V(F_\ell^\ast)=V(H)$,
			\item \label{blowup_4} $V(F_i^\ast)\cap V(F_j^\ast)=\emptyset$ if $\abs{i - j}\geq 2$, and
			\item \label{blowup_5} $V(F_j^\ast)\cap V(F_{j+1}^\ast)$ has size $k$ and induces an edge in $F_j^\ast$ and $F_{j+1}^\ast$ which is disjoint with the singleton parts of $F_j^\ast$ and $F_{j+1}^\ast$ \textup{(}if they exist\textup{)}.
		\end{enumerate}
	\end{lemma}
	
	Suppose we have a sequence $F_1^\ast, \dotsc, F_\ell^\ast$ as in the statement of \cref{lem:findblowupchain}.
	We would be done if we could cover each $F_i^\ast$ with a spanning sphere, while making sure that the edges induced by $V(F_{i-1}^\ast)\cap V(F_{i}^\ast)$ and $V(F_{i}^\ast)\cap V(F_{i+1}^\ast)$ are facets.
	In fact, then we would simply glue these spheres along the common facet and get a spanning sphere of $H$.
	This will be done with our second key lemma.
	
	\begin{lemma}[Allocation]\label{lem:allocation}
		Let $1/m\ll 1/s\ll \gamma \ll \eps, 1/k\leq 1/3$, and let $R$ be an $s$-vertex $k$-graph without isolated vertices and with $\comin(R) \ge (1/2 + \eps)s$. Let $R^\ast$ be a $(\gamma, m)$-nearly-regular blow-up of $R$ and let $f_1, f_2 \in E(R^\ast)$ such that $\phi(f_1)$, $\phi(f_2)$ and the vertex in $R$ corresponding to the singleton part of $R^\ast$ \textup{(}if it exists\textup{)} are all disjoint. Then, $R^\ast$ contains a spanning copy of $\bS^{k - 1}$ where $f_1$ and $f_2$ are facets.
	\end{lemma}

	\subsection{Proof of main result}
	
	Assuming \cref{lem:findblowupchain,lem:allocation}, we can easily prove our main result.
	
	\begin{proof}[Proof of \cref{thm:maintheorem}]
		Fix constants satisfying $1/n\ll 1/m_2\ll 1/m_1 \ll 1/s \ll \gamma \ll \eps, 1/k\leq 1$ and apply \cref{lem:findblowupchain} to find $F_1,\dotsc, F_\ell$ and $F_1^\ast,\dotsc, F_\ell^\ast\subseteq H$ with properties as listed in the statement of \cref{lem:findblowupchain}. For each $i\in [2,\ell-1]$, apply \cref{lem:allocation} to $F_i^\ast$ (with $\eps/2$ playing the role of $\eps$) with $f_1, f_2$ being the edge induced by $V(F_{i-1}^\ast)\cap V(F_{i}^\ast), V(F_{i}^\ast)\cap V(F_{i+1}^\ast)$, respectively (note that $f_1$ and $f_2$ are disjoint by property~\ref{blowup_4}). Apply \cref{lem:allocation} to $F_1^\ast$ with $f_2$ being the edge induced by $V(F_{1}^\ast)\cap V(F_{2}^\ast)$, and~$f_1$ set to be an edge disjoint with $f_2$ and the singleton part, if it exists (here, the choice does not matter). Similarly, apply \cref{lem:allocation} to $F_\ell^\ast$ with $f_1$ being the edge induced by $V(F_{\ell-1}^\ast)\cap V(F_{\ell}^\ast)$, and $f_2$ chosen arbitrarily to be disjoint with $f_1$ and the singleton part, if it exists.
		
		By gluing together the spanning simplicial $(k - 1)$-spheres produced by each of these applications of \cref{lem:allocation} (using \cref{rmk:gluecommonface}), we get a spanning simplicial $(k - 1)$-sphere of~$H$, as desired.
	\end{proof}
	
	\subsection{Lower bound constructions}\label{sec:constructions}
	
	The following construction due to Georgakopoulos et al.~\cite{georgakopoulos2022spanning} shows that the degree condition in \cref{conj:mainconj} is tight for $k=3$.
	Let $H$ be an $n$-vertex $3$-graph where $V(H) = \set{u,v}\cup X \cup Y$ where $\abs{X} = \abs{Y} = (n-2)/2$ and $E(H)$ consists of all possible $3$-edges with the exception of those meeting both $X$ and $Y$. It is easy to check that $\comin(H) = n/2-O(1)$ and, furthermore, $H$ does not contain a spanning copy of a $2$-sphere. Indeed, if it did, removing every edge that contains both $u$ and $v$ from this copy splits the copy of the sphere into two tight components, which gives a contradiction using elementary topological arguments.
	
	This construction naturally generalises to $k$-graphs by replacing $\set{u,v}$ with a set of $k - 1$ vertices.
	{Indeed, let $T = \set{u_1, \dotsc, u_{k - 1}}$ be the set of $k - 1$ vertices that replaces $\set{u, v}$ and suppose there is a spanning $(k - 1)$-sphere $S$ in $H$. Note that at most two edges of $S$ contain~$T$. These edges form a $(k - 1)$-ball with no vertices in its interior. Since a $(k - 1)$-sphere is $(k - 2)$-connected in the topological sense, the removal of these edges does not disconnect $S$. However, removing all edges that contain $T$ from $H$ splits $H$ into two tight components.}
	
	\section{Tools}\label{sec:tools}
	
	\subsection{Connectivity}
	
	A $k$-uniform \defn{tight walk} $W$ in a $k$-graph $G$ comes with an ordered multiset of vertices such that the edges of $W$ are precisely the $k$-sets of consecutive vertices. A \defn{tight path} is a tight walk that does not repeat vertices.
	The following result~\cite[Prop.~5.1]{LS23} ties tight connectivity (recall this definition from the introduction) and tight walks. We include the proof for completeness.
	
	\begin{lemma}\label{lem:tight-connectivity-walk}
		Let $H$ be a $k$-graph. Then $H$ is tightly connected if and only if $H$ has no isolated vertices and contains a tight walk that contains every edge of $H$ as a subwalk.
	\end{lemma}
	
	\begin{proof}
		Consecutive edges in tight walks have at least $k - 1$ vertices in common and so tight walks are tightly connected. This proves the `if' direction. Now suppose that $H$ is tightly connected. We will show that, for every $ef \in E(\hat{H})$ and every tight walk $W'$ in $H$ that contains $e$, there is a tight walk $W$ in $H$ that covers all the edges of $W'$ and in addition covers~$f$. Since $\hat{H}$ is connected, repeatedly applying this result will give a tight walk in $H$ that contains every edge of $H$.
		
		Since $W'$ contains $e$, there is a subwalk $W_1 = (x_1, \dotsc, x_k)$ of $W'$ such that $e = \set{x_1, \dotsc, x_k}$. By assumption $\abs{e \cap f} = k - 1$ and so $e \setminus f = \set{x_j}$ for some $1 \leq j \leq k$ and $f \setminus e = \set{y}$ for some vertex $y$. It follows that
		\begin{equation*}
			W_2 = (x_1, \dotsc, x_k, x_1, \dotsc, x_{j - 1}, y, x_{j + 1}, \dotsc, x_k, x_1, \dotsc, x_k)
		\end{equation*}
		is a tight walk that starts and ends at $e$ and visits $f$. Replacing $W_1$ by $W_2$ in $W'$ gives the required tight walk $W$.
	\end{proof}
	
	The \defn{order} of a tight walk $W$ is the number of vertices in the multiset. We now bound the order of the tight walk given by the previous lemma.
	
	\begin{lemma}\label{lem:tight-connectivity-bounded-tight-walk}
		Let $H$ be an $n$-vertex tightly connected $k$-graph.
		Then $H$ contains a tight walk of order at most $n^{2k}$ that contains each edge of $H$ as a subwalk.
	\end{lemma}
	
	\begin{proof}
		By \cref{lem:tight-connectivity-walk}, there is a tight walk $W$ that visits all edges of $H$.
		Suppose that $W$ has minimal order with this property.
		Let $e_1, \dotsc, e_r$ be an enumeration of the edges of $H$ according to the first time that they are visited in $W$.
		Denote by $W_i$ the subwalk of $W$ that starts with the first appearance of $e_i$ and ends just before the first appearance of $e_{i + 1}$.
		So we can write $W$ as the concatenation of $W_1 \dotsc W_r$.
		Note that within every $W_i$ each tuple~$f$ of $k$ vertices may only appear once.
		Otherwise, some subwalk of $W_i$ starting with $f$ and ending just before another occurrence of $f$ could be removed from $W$, while retaining the property that every edge is visited.
		Since there are at most $\binom{n}{k}$ edges and at most $n^k$ tuples of $k$ vertices, the claim follows.
	\end{proof}

	\begin{lemma}[Tightly connected]\label{lem:dirac-to-tightly-connected}
		Every $n$-vertex $k$-graph with $\comin(H) \geq \floor{(n - k + 1)/2}$ and no isolated vertices is tightly connected.
	\end{lemma}
	
	\begin{proof}
		We show that if $e$ and $f$ are distinct edges of $H$, then there are edges $e'$ and $f'$ such that $\abs{e \cap e'}, \abs{f \cap f'} \geq k - 1$ and $\abs{e' \cap f'} > \abs{e \cap f}$. Iteratively applying this shows that $e$ and $f$ are in the same tight component of $\hat{H}$, and so $H$ is tightly connected.
		
		Write $e = u_1 \dotsc u_\ell v_{\ell + 1} \dotsc v_k$ and $f = u_1 \dotsc u_\ell w_{\ell + 1} \dotsc w_k$ where $\ell = \abs{e \cap f}$. Let $S = \set{u_1, \dotsc, u_\ell, v_{\ell + 1}, \dotsc, v_{k - 1}}$ and $T = \set{u_1, \dotsc, u_\ell, w_{\ell + 1}, \dotsc, w_{k - 1}}$. Note that $S, T \in \partial E(H)$.
		
		Let $\Gamma(S) = \set{x \in V(H) \colon S \cup \set{x} \in E(H)}$ and $\Gamma(T) = \set{x \in V(H) \colon T \cup \set{x} \in E(H)}$. If $\Gamma(S) \cap f \neq \emptyset$, then some $w_i \in \Gamma(S)$ and so $e' = S \cup \set{w_i}$ and $f' = f$ satisfy the conditions claimed. Similarly if $\Gamma(T) \cap e \neq \emptyset$. Thus we may assume that $\Gamma(S), \Gamma(T) \subseteq V(H) \setminus (e \cup f)$.
		
		Now $\abs{V(H) \setminus (e \cup f)} \leq n - (k + 1)$ and $\abs{\Gamma(S)} + \abs{\Gamma(T)} \geq 2\comin(H) \geq 2\floor{(n - k + 1)/2} > n - (k + 1)$, and so there is some vertex $x \in \Gamma(S) \cap \Gamma(T)$. Then $e' = u_1 \dotsc u_\ell v_{\ell + 1} \dotsc v_{k - 1} x$ and $f' = u_1 \dotsc u_\ell w_{\ell + 1} \dotsc w_{k - 1} x$ are edges satisfying the conditions claimed.
	\end{proof}
	
	\subsection{Minimum supported \texorpdfstring{$d$}{d}-degree}
	
	For $1 \leq d < k$ and a $k$-graph $H$, the \defn{minimum $d$-degree} of $H$, denoted by $\delta_d(H)$, is the maximum integer $m$ such that every set of $d$ vertices has degree at least $m$ in $H$.
	We write $\partial_d H$ for the set of supported $d$-sets in $H$.
	For non-empty~$H$, the \defn{minimum supported $d$-degree}, denoted by $\comin_d(H)$, is the maximum integer $m$ such that every supported $d$-set has degree at least $m$ in $H$.
	If $H$ is empty, we set $\comin_d(H)=0$.
	We remark that $\delta_{k - 1}(H) = \delta(H)$ and $\comin_{k - 1}(H)=\comin(H)$.
	
	Analogously to the usual minimum $d$-degree, the notion of minimum supported $d$-degree is stronger for larger $d$.
	This is formalised in the following fact.
	
	\begin{fact}\label{claim:supported_minimum_degree}
		Let $H$ be a $k$-graph and $d$ be an integer with $1 \le d < k - 1$.
		Then $\comin_d(H) \ge \frac{\comin(H)}{k-d} \cdot \comin_{d+1}(H)$.
	\end{fact}
	
	\begin{proof}
		The result follows from a double counting argument.
		Let $S$ be a supported $d$-subset of $V(H)$ and define $\cF = \set{(e,v) \colon e \in E(H), e \supset S \text{ and } v \in e \setminus S}$.
		
		There are $\deg_H(S)$ ways to choose $e \in E(H)$ with $e \supset S$ and, having fixed any such $e$, there are $k - \abs{S}$ ways to choose $v \in e \setminus S$.
		Therefore $\abs{\cF} = \deg_H(S) \cdot (k-d)$.
		
		Since $S$ is supported, there are at least $\comin(H)$ distinct vertices $v \in V(H) \setminus S$ such that $S \cup \set{v}$ is supported.
		For any such $v$, there are at least $\comin_{d+1}(H)$ choices of $e \in E(H)$ such that $e \supseteq S \cup \set{v}$.
		Therefore $\abs{\cF} \ge \comin(H) \cdot \comin_{d+1}(H)$.
		
		In particular, $\deg_H(S) \ge \frac{\comin(H)}{k-d} \cdot \comin_{d+1}(H)$.
		Since this is true for any supported $d$-set $S$, the result follows.
	\end{proof}
	
	\subsection{Concentration}
	
	We use the following standard concentration bound.
	\begin{lemma}[{\cite[Cor.~2.2]{concentration}}]\label{lem:concentration}
		Let $V$ be an $n$-set with a function $h$ from the $s$-sets of $V$ to $\bR$.
		Suppose that there exists $K \geq 0$ such that $\abs{h(S)-h(S')} \le K$ for any $s$-sets $S, S' \subset V$ with $\abs{S \cap S'} = s-1$.
		Let $S \subset V$ be an $s$-set chosen uniformly at random.
		Then, for any $\ell >0$,
		\begin{equation*}
			\Prob(\abs{h(S) - \Exp [h(S)]} \geq \ell) \leq 2 \exp\biggl(-\frac{2 \ell^2}{\min\set{s, n-s} K^2}\biggr).
		\end{equation*}
	\end{lemma}
	
	\subsection{Property graph} \label{sec:property_graph}
	
	Given a property $\cP$ and a graph $H$ satisfying $\cP$, we are interested in which subgraphs of $H$ inherit the property $\cP$. Following~\cite{lang2023tiling}, this is formalised in terms of the property graph.
	
	\begin{definition}[Property graph] \label{def:property-graph}
		For an $m$-graph $H$ and a family of $s$-vertex $m$-graphs $\cP$, the \defn{property graph}, denoted by $\PG{H}{\cP}{s}$, is the $s$-graph on vertex set $V(H)$ with an edge $S \subset V(H)$ whenever the induced subgraph $H[S]$ \defn{satisfies}~$\cP$, that is $H[S] \in \cP$.
	\end{definition}
	
	We use the following result, which appeared in the context of hypergraph tilings~\cite[Lemma~4.4]{lang2023tiling}.
	For sake of completeness, its proof can be found in \cref{lem:almost-blow-up-cover}.
	
	\begin{lemma}[Almost perfect blow-up-tiling] \label{lem:covering-with-blow-ups}
		For all $2\leq k \leq s$, $m\geq 1$ and $\mu,\eta >0$, there is an $n_0>0$ such that the following holds for every $s$-vertex $k$-graph property $\cP$ and $k$-graph $H$ on $n \geq n_0$ vertices  with
		\begin{equation*}
			\delta_1 \bigl(\PG{H}{ \cP }{s}\bigr) \geq   (1 - 1/s + \mu) \tbinom{n-1}{s-1}.
		\end{equation*}
		All but at most $\eta n$ vertices of $H$ may be covered with pairwise vertex-disjoint $m$-regular blow-ups of members of $\cP$.
	\end{lemma}
	
	We define \defn{$\cP(\eps, k)$} to be the family of $k$-graphs $H$ with $\comin(H) \geq (1/2 + \eps)\abs{V(H)}$ and without any isolated vertices.
	The next lemma shows that $\cP(\eps,k)$ satisfies a local inheritance principle.
	
	\begin{lemma}\label{lem:robustmoregeneral}
		Let $1/n\ll 1/s \ll 1/k, 1/t, \eps$. Let $H$ be an $n$-vertex $k$-graph such that $H\in \cP(\eps,k)$, and let $T\subseteq V(H)$ be a $t$-subset. Let $S$ be an $(s-t)$-subset of $V(H) \setminus T$ chosen uniformly at random.
		Then, with probability at least $1-e^{-\sqrt{s}}$, we have that $H[T\cup S]\in \cP(\eps/2,k)$.
	\end{lemma}
	
	\begin{proof}
		Let $S$ be an $(s-t)$-subset of $V(H) \setminus T$ chosen uniformly at random and $S'=T \cup S$.
		Observe that $S'$ is distributed as a uniformly randomly chosen set among all $s$-sets that include $T$.
		Fix $d \in \set{1,k - 1}$, let $D \subseteq V(H)$ be a supported $d$-set and set $h(S') = \deg_{H[S']}(D)$.
		Using \Cref{lem:concentration}, we show that $h(S')$ is concentrated around its expectation.
		We give all the calculations for $d=k - 1$, but we omit those for $d=1$, as they are analogous.

		Let $d=k - 1$ and condition on the event that $D \subseteq S'$.
		For a fixed edge $e \in E(H)$ with $e \cap (T \cup D) = D$, the probability that $e \subset S'$ is
		\begin{align}
			\label{eq:probability_e_S'}
			\begin{split}
				\binom{n - \abs{T \cup D \cup e}}{s - \abs{T \cup D \cup e}} \binom{n - \abs{T \cup D}}{s - \abs{T \cup D}}^{-1}
				& = \frac{(s)_{q}}{(n)_{q}} \binom{n-k}{s-k} \biggl(\frac{(s)_{q}}{(n)_{q}}\binom{n-d}{s-d}\biggr)^{-1} \\
				& = \binom{n-k}{s-k} \binom{n-d}{s-d}^{-1} = \binom{s-d}{k-d} \binom{n-d}{k-d}^{-1}\, ,
			\end{split}
		\end{align}
		where $q = \abs{T \setminus D}$ and $(s)_q, (n)_q$ denote the falling factorials.
		By assumption, there are at least $(1/2 + \eps) n - q$ edges $e \in E(H)$ with $e \cap (T \cup D) = D$, and thus $\Exp [h(S')] \geq  (1/2 + (3/4) \eps) s$.
		
		Moreover, for any two $s$-sets $R,R' \subset V(H)$  with $\abs{R \cap R'} = s - 1$ that both contain $T \cup D$, we have $\abs{h(R) - h(R')} \leq 1$.
		So by \cref{lem:concentration} applied with $\ell = (\eps/4) s$ and $K = 1$, we have
		\begin{align*}
			\Prob\bigl[h(S) < (1/2 + \eps/2) s \bigr] & \leq \Prob \bigl[\Exp[h(S)] - h(S) \geq  (\eps/4) s \bigr]                                                                             \\
			& \leq 2 \exp\biggl(-\frac{\eps^2 s^2}{8} \cdot \frac{1}{s - \abs{T \cup D}}\biggr) \leq \frac{1}{2} \binom{s}{d}^{-1} \exp(-\sqrt{s}) .
		\end{align*}
		Therefore by the union bound over the supported $d$-subsets $D \subseteq S$, we get that $\comin(H[S']) \ge (1/2+\eps/2)s$ with probability at least $1-\frac{1}{2}e^{-\sqrt{s}}$.
		
		Since $H \in \cP(\eps,k)$, $H$ does not have isolated vertices.
		In particular, every vertex is supported and, by \Cref{claim:supported_minimum_degree}, is contained in at least $\frac{\comin(H)}{k - 1} \cdot \frac{\comin(H)}{k-2} \cdot \dots \cdot \frac{\comin(H)}{2} \cdot \comin(H) = \Omega_k(n^k)$ edges of $H$.
		Therefore we can proceed as above and prove that, with probability at least $1-\frac{1}{2}e^{-\sqrt{s}}$, every vertex of $S'$ is contained in at least one edge of $H[S']$.
		
		We conclude that $H[S']=H[T \cup S] \in \cP(\eps/2,k)$ with probability at least $1-e^{-\sqrt{s}}$, as desired.
	\end{proof}
	
	The $t=1$ case of \cref{lem:robustmoregeneral} has the following important corollary.
	
	\begin{corollary}[Property graph is robust]\label{cor:propgraphrobust}
		Let $1/n\ll 1/s \ll 1/k, \eps$. Let $H$ be an $n$-vertex $k$-graph such that $H\in \cP(\eps,k)$. Then,
		\begin{equation*}
			\delta_1 \bigl(\PG{H}{ \cP(\eps/2,k) }{s} \bigr) \geq \bigl(1-1/s^2 \bigr) \tbinom{n-1}{s-1}.
		\end{equation*}
	\end{corollary}

	\section{Blow-up chains}\label{sec:blow-up}
	
	This section is dedicated to the proof of \cref{lem:findblowupchain}, for which we need a few preliminary results concerning blow-ups. The first one is a well-known insight of Erd\H{o}s~\cite{Erdos1964hypextremal}, stating that the Tur\'an density of $K_s^{(s)}(b)$ is zero, where \defn{$K_s^{(s)}(b)$} denotes the complete $s$-partite $s$-graph with each part of size $b$.
	
	\begin{theorem}
		\label{thm:erd64}
		For all $s\geq 2$, $b \geq 1$ and $\gamma > 0$, there is $n_0>0$ such that every $s$-graph $P$ on $n\geq n_0$ vertices with $e(P) \geq \gamma n^s$ contains a copy of $K_s^{(s)}(b)$.
	\end{theorem}
	
	Next, we present two applications of \cref{thm:erd64}.
	For a partition $\cV = \set{V_1, \dotsc, V_s}$ of a ground set $V$, a subset $A \subset V$ is called \defn{$\cV$-partite} if it has at most one element in each part.
	A $k$-graph $H$ on vertex set $V$ is \defn{$\cV$-partite} if all its edges are $\cV$-partite.
	We often do not explicitly mention the partition $\cV$ and just speak of an \defn{$s$-partite} graph $H$.
	Finally a blow-up $F^\ast \subset H$ of a $k$-graph $F$ is called \defn{consistent} in the $\cV$-partite $H$ if the parts of $F^\ast$ can be written $U_1, \dotsc, U_s$ such that $U_i \subset V_i$ for each $i \in [s]$.
	
	\begin{lemma}\label{lem:pigeonhole}
		Let $1/m_2 \ll 1/m_1, 1/s, 1/k$.
		Let $\cP$ be a family of $s$-vertex $k$-graphs. Let $\cA$ be an $s$-partite $k$-graph with parts $A_1, \dotsc, A_s$ where $\abs{A_i} = m_2$ for all $i\in[s]$. Suppose that every $\set{A_i}_{i \in [s]}$-partite $s$-set induces a member of $\cP$ in $\cA$. Then, there exists some $P \in \cP$ so that $\mathcal{A}$ contains a consistent $m_1$-regular blow-up of $P$.
	\end{lemma}
	
	\begin{proof}
		Colour every $\set{A_i}_{i \in [s]}$-partite $s$-set $S$ by a pair $(P, \phi)$ where $P \in \cP$ and the function $\phi \colon \set{A_1, \dotsc, A_s} \to V(P)$ is induced by the isomorphism between $P$ and the $k$-graph induced by $S$. Note that this is possible by the hypotheses of the lemma. Observe that the number of such pairs is at most $\abs{\cP} \cdot s! \leq 2^{s^k}s!$, so, by averaging, we may find a monochromatic subgraph of $\mathcal{A}$ with at least $(2^{s^k}s!)^{-1} m_2^{s} = (2^{s^k} s! s^s)^{-1} (s m_2)^{s}$ edges. Applying \cref{thm:erd64} with $\gamma = (2^{s^k} s! s^s)^{-1}$ and $b = m_1$ we can guarantee the existence of the desired $m_1$-regular blow-up (as $1/m_2\ll 1/m_1, 1/s$).
	\end{proof}
	
	\begin{lemma}[Rooted blow-ups]\label{lem:rooted-blow-ups}
		Let $1/n \ll 1/m_2 \ll 1/m_1, 1/s,1/k$, $0 \leq q < s$ and $\mu>0$.
		Let~$G$ be an $n$-vertex $k$-graph.
		Let $\cV = \set{V_i}_{i \in [q]}$ be a family of pairwise disjoint $m_2$-sets in $V(G)$.
		Let $\cP$ be a family of $s$-vertex $k$-graphs.
		Suppose that, for every $\cV$-partite $q$-set $X$, there are at least $\mu n^{s - q}$ subsets $S \subseteq V(G) \setminus \big(X \cup V_1 \cup \dots \cup V_q\big)$ of size $s - q$ such that $G[S \cup X] \in \cP$.
		Then there is a family $\cU = \set{U_i}_{i \in [s - q]}$ of pairwise disjoint $m_1$-sets disjoint with $\bigcup_{i\in[q]} V_i$ such that the $s$-partite $k$-graph induced by $\cV \cup \cU$ contains a consistent $m_1$-regular blow-up of some $P\in \cP$.
	\end{lemma}
	
	\begin{proof}
		We begin with the following claim which follows from a simple first moment argument.
		\begin{claim}
			There is a family $\cU' = \set{U_i'}_{i \in [s - q]}$ of pairwise disjoint $m_2$-sets disjoint with $\bigcup_{i\in[q]} V_i$ such that at least $\mu m_2^s$ $(\cV \cup \cU')$-partite $s$-sets induce a member of $\cP$ in $G$.
		\end{claim}
		\begin{proof}
			Let $V'=V(G) \setminus (V_1 \cup \dots \cup V_q)$ and $U_1', \dots, U_{s - q}'$ be defined recursively as follows: for each $i \in [s - q]$, $U_i'$ is chosen uniformly at random among all $m_2$-subsets of $V' \setminus \big(U_1' \cup \dots \cup U_{i-1}'\big)$.
			Then denote $\cU' = \set{U_i'}_{i \in [s - q]}$.
			
			Fix a subset  $S'\subseteq V(G)$ of the form $S'=X \cup S$, where $X$ is a $\cV$-partite $q$-set and $S \subseteq V'$ (of size $s - q$).
			Observe that $S'$ is $(\cV \cup \cU')$-partite if and only if $S$ is $\cU'$-partite.
			Therefore
			\begin{align*}
				\mathbb{P}\big[S' \text{ is } (\cV \cup \cU')\text{-partite}\big] \ge & \, \frac{\binom{\abs{V'}-(s - q)}{m_2-1}}{\binom{\abs{V'}}{m_2}} \cdot
				\frac{\binom{\abs{V'}-m_2-(s - q-1)}{m_2-1}}{\binom{\abs{V'}-m_2}{m_2}} \cdot                                                                  \\
				& \cdot
				\frac{\binom{\abs{V'}-2m_2-(s - q-2)}{m_2-1}}{\binom{\abs{V'}-2m}{m_2}} \cdot
				\dotsc
				\cdot
				\frac{\binom{\abs{V'}-(s - q-1)m_2-1}{m_2-1}}{\binom{\abs{V'}-(s - q)m_2}{m_2}}                                                                \\
				=                                                                     & \, m_2^{s - q} \cdot \frac{(\abs{V'}-(s - q))!}{\abs{V'}!}\, ,
			\end{align*}
			where the first expression is the probability that, for a fixed arbitrary ordering of the elements of $S$, say $v_1, \dots, v_{s - q}$, we have that among all the elements of $S$, the set $U_i'$ contains $v_i$ only.
			
			By assumption, there are at least $m_2^q \cdot \mu n^{s - q}$ subsets $S'$ of the form above such that $G[S'] \in \cP$.
			Then by linearity of expectation, the number of such subsets $S'$ which in addition are $(\cV \cup \cU')$-partite is at least
			\begin{equation*}
				m_2^q \cdot \mu n^{s - q} \cdot m_2^{s - q} \cdot \frac{(\abs{V'}-(s - q))!}{\abs{V'}!} \ge \mu m_2^s\, .
			\end{equation*}
			where we used that $\abs{V'} \le n$.
			We conclude that there is a choice of $\cU'$ with the desired properties.
		\end{proof}
		Let $H'$ be the $s$-partite $s$-graph with parts $V_1, \dotsc, V_q, U'_1, \dotsc, U'_{s - q}$ where a $(\cV \cup \cU')$-partite $s$-set is an edge if it induces a member of $\cP$ in $G$.
		Fix a new constant $m_{1.5}$ such that $1/m_2\ll 1/m_{1.5}\ll 1/m_1$. By \cref{thm:erd64} (applied to $H'$ with $\gamma = \mu$ and $b = m_{1.5}$), $H'$ contains a copy of $K_s^{(s)}(m_{1.5})$.
		Therefore there exist $U_i \subseteq U_i'$ for each $i \in [s - q]$ and $V_j' \subseteq V_j$ for each $j \in [q]$ such that $\abs{U_i} = \abs{V_j'} = m_{1.5}$ for each $i$ and $j$, and every $(\set{U_i}_{i \in [s - q]} \cup \set{V'_j}_{j \in [q]})$-partite $s$-set induces a member of $\cP$ in $G$. \Cref{lem:pigeonhole} then implies the desired result with $\cU = \set{U_i}_{i \in [s - q]}$.
	\end{proof}

	We are now ready to prove the main result of this section, \cref{lem:findblowupchain}.
	
	\begin{proof}[Proof of \cref{lem:findblowupchain}]
		Introduce constants $\eta$ and $m_{1.5}$ with $1/n \ll \eta \ll 1/m_2 \ll 1/m_{1.5} \ll 1/m_1$. Apply \cref{lem:covering-with-blow-ups} with $1/s-1/s^2$, $m_2$, $\mathcal{P}(\eps/2, k)$ playing the roles of $\mu$, $m$, $\mathcal{P}$, respectively (the other variables have identical names), noting the hypothesis on $\delta_1$ holds directly by \cref{cor:propgraphrobust}. This gives us a collection of $s$-vertex $k$-graphs $R_1,\dotsc, R_\ell\in \mathcal{P}(\eps/2, k)$ such that there are pairwise vertex-disjoint $m_2$-regular blow-ups of each of the $R_i$, say $R_i^\ast(m_2)$, which together cover all but $\eta n$ vertices of $H$. We denote by $L$ the set of the uncovered vertices.
		
		Next, we find vertex-disjoint blow-ups that cover the vertices of $L$.
		\begin{claim}
			\label{claim:saturatingL}
			There exist $s$-vertex $k$-graphs $L_1,\dotsc, L_{\ell'}\in \mathcal{P}(\eps/2, k)$ such that $H$ contains vertex-disjoint $(0,m_{1.5})$-nearly-regular blow-ups of each of the $L_i$, say $L_i^\ast(m_{1.5})$, which together cover $L$, and furthermore, for each $i\in [\ell]$, $R_i^\ast(m_2)\setminus \bigcup_{j\in [\ell']} L_j^\ast(m_{1.5})$ is a $(\gamma/2, m_2)$-regular blow-up of $R_i$.
		\end{claim}
		\begin{proof}
			We will achieve this by greedily saturating each $v\in L$ one by one with suitable blow-ups, never using more than $(\gamma/2)m_2$ vertices from any of the parts of the blow-ups $R_1^\ast(m_2), \dots, R_{\ell}^\ast(m_2)$.
			Observe that this immediately implies that  $R_i^\ast(m_2)\setminus \bigcup_{j\in [\ell']} L_j^\ast(m_{1.5})$ is a $(\gamma/2, m_2)$-regular blow-up of $R_i$ as needed.
			
			Suppose we proceed with the above plan for as long as possible, and suppose that there remains an uncovered vertex $v\in L$.
			Call a vertex of $H$ \defn{unavailable} if it has already been used by other blow-ups to cover the vertices of $L\setminus \set{v}$ and observe that, using $\abs{L} \le \eta n$, there are at most $\eta n(sm_{1.5}) \leq \sqrt{\eta}n$ unavailable vertices, where we used that $\eta \ll 1/m_{1.5}$.
			Moreover, using $\eta \ll \gamma$, among all parts of the blow-ups $R_1^\ast(m_2), \dots, R_{\ell}^\ast(m_2)$, at most $\frac{\eta n (sm_{1.5})}{\gamma m_2/4} \le  \frac{\sqrt{\eta} n(sm_{1.5})}{m_2}$ contain at least $\gamma m_2/4$ unavailable vertices, and we call the vertices that are contained in such parts \defn{dangerous}.
			Note that, since each part has size $m_2$, there are then at most $\sqrt{\eta}n (sm_{1.5})$ dangerous vertices.
			\fw{In particular, using $\eta\ll 1/m_{1.5} \ll 1/s$, there are at most $\sqrt{\eta} n + \sqrt{\eta}n(sm_{1.5}) \le \eta^{1/4} n$ vertices which are unavailable or dangerous.}
			\fw{We conclude that no more than $\eta^{1/4} n \cdot n^{s-2} = \eta^{1/4} n^{s-1}$ distinct $(s-1)$-subsets of $V(H)$ contain a vertex which is dangerous or unavailable.}

			By \cref{lem:robustmoregeneral} applied with $T=\set{v}$,
			there exist at least $(1-e^{-\sqrt{s}})\binom{n-1}{s-1}$ many distinct $(s-1)$-subsets $S$ of $V(H) \setminus \set{v}$ with $H[\set{v} \cup S]\in \mathcal{P}(\eps/2,k)$. Denote by $Q$ the auxiliary $(s-1)$-graph on $V(H) \setminus \set{v}$ where, for an $(s-1)$-subset $S \subseteq V(H) \setminus \set{v}$, we have $S \in E(Q)$ if and only if $S$ does not contain any dangerous or unavailable vertex and $H[\set{v} \cup S]\in \mathcal{P}(\eps/2,k)$.
			\fw{Then, by the above calculations, $e(Q) \ge (1-e^{-\sqrt{s}}) \binom{n-1}{s-1} - \eta^{1/4} n^{s-1} \ge n^{s-1}/s!$ and by \cref{thm:erd64} (applied with $\gamma=1/s!$ and $m=m_2$), we find a copy $K$ of $K_{s-1}^{(s-1)}(m_2)$ inside $Q$.}
			
			Consider the auxiliary $s$-partite $k$-graph $\mathcal{A}$ obtained from $H$ as follows: the parts are the $(s-1)$ parts of $K$ and the last part is given by the vertex $v$ blown-up to $m_2$ new vertices.
			Then $\mathcal{A}$ satisfies the hypothesis of \cref{lem:pigeonhole} (with $\mathcal{P}=\mathcal{P}(\eps/2,k)$): indeed, it is $s$-partite, each part has size $m_2$ and each $s$-partite edge of $\mathcal{A}$ induces a member of $\mathcal{P}$ by definition of $Q$.
			Therefore, $\mathcal{A}$ contains a consistent $m_{1.5}$-regular blow-up of some $P \in \mathcal{P}$, which corresponds to a $(0,m_{1.5})$-nearly-regular blow-up in $H$ of the desired form saturating~$v$. Indeed, this blow-up uses at most $m_{1.5}s \le \gamma m_2/4$ vertices and thus the union of all the blow-ups does not use more than $\gamma m_2/4+\gamma m_2/4=\gamma m_2/2$ vertices from any part of the blow-ups $R_1^\ast(m_2), \dots, R_{\ell}^\ast(m_2)$, where the argument is valid as the blow-up saturating $v$ does not use any dangerous vertex by definition of $Q$. Therefore we can extend the blow-up collection and cover $v$, a contradiction.
		\end{proof}
		
		Redefine $R_i^\ast(m_2) \coloneqq R_i^\ast(m_2)\setminus \bigcup_{j\in [\ell']} L_j^\ast(m_{1.5})$. Observe also that we can partition each $R_i^\ast(m_2)$ into $\Theta(m_2/m_{1.5})$ parts, each part being a $(\gamma, m)$-regular blow-up of $R_i$ for some $m\in [m_{1.5},2m_{1.5}]$. For convenience, we consider an arbitrary linear ordering $B_1^\ast, \dotsc, B_t^\ast$ of all the blow-ups we have constructed so far and observe that such sequence satisfies properties \ref{blowup_1} through \ref{blowup_3} of the lemma with $\gamma, m_{1.5}, 2m_{1.5}$ in place of $\gamma, m_1,m_2$, and trivially property~\ref{blowup_4} holds as well since the blow-ups are pairwise vertex-disjoint.
		
		It remains to find some connecting blow-ups so that we can ensure~\ref{blowup_5} as well. For a given~$B_i^*$, we denote by $B_i$ the underlying graph (which is an $s$-vertex graph from $\mathcal{P}(\eps/2,k)$).
		
		\begin{claim}\label{claim:glue}
			There exists $s$-vertex $k$-graphs $C_1,\dotsc, C_t \in \mathcal{P}(\eps/2, k)$ such that $H$ contains vertex-disjoint $m_1$-regular blow-ups of each of the $C_i$, say $C_i^\ast(m_1)$, so that for each $i\in [t-1]$
			$V(B_i^\ast)\cap  V(C_i^\ast(m_1))$ \textup{(}respectively,  $V(C_i^\ast(m_1))\cap  V(B_{i+1}^\ast)$\textup{)} has size $k$, is disjoint with their singleton parts \textup{(}if they exist\textup{)} and induces an edge in $B_i^\ast$ and $C_i^\ast$ \textup{(}respectively, $C_i^\ast$ and $B_{i+1}^\ast$\textup{)}. Furthermore, for each $i\in [t]$, $B_i^\ast\setminus \bigcup_{j\in [t-1]} C_j^\ast(m_1)$ is a $(2\gamma, m)$-regular blow-up of $B_i$ for some $m\in [m_1,m_2]$.
		\end{claim}
		
		The proof is similar to the proof of \cref{claim:saturatingL}.
		\begin{proof}[Proof of \cref{claim:glue}]
			Define a new constant $\zeta$ with $1/m_1\ll \zeta \ll 1/s,\gamma$.
			For each $i\in [t]$, pick an edge in $B_i$ and let $b_{\text{left}}^i$ be the corresponding set of $k$ parts of $B_i^\ast$.
			Since  $B_i^\ast$ has at most one `exceptional' part of size exactly $1$, we can assume that the parts of  $b_{\text{left}}^i$ all have size $(1\pm\gamma)m$ for some $m \in [m_{1.5},2m_{1.5}]$.
			We pick $b_{\text{right}}^i$ in the same way, but disjoint of $b_{\text{left}}^i$.\footnote{We remark that, for $i=1$, we will only need $b_{\text{right}}^i$ and, for $i=t$, we will only need $b_{\text{left}}^i$.}
			We will find inductively $C_1, \dotsc, C_{j-1}$ and corresponding blow-ups as in the statement of the claim while additionally ensuring that
			\begin{enumerate}[label = (\alph*)]
				\item  \label{item:claim_dangerous} for each $i\in [t]$, the blow-ups $C_1^\ast(m_1), \dots, C_{j-1}^\ast(m_1)$ do not use more than $\gamma$-fraction of each part of $B_i^\ast$.
			\end{enumerate}
			
			Suppose this has been done for $C_1, \dots, C_{j-1}$.
			\fw{Call a vertex $v$ \defn{special} if it belongs to $B_j^\ast \cup B_{j+1}^\ast$,} call $v$ \defn{unavailable} if it has already been used for one of the blow-ups of $C_1, \dots, C_{j-1}$ and call $v$ \defn{dangerous} if the fraction of unavailable vertices in the part of the blow-ups $B_1^\ast, \dots, B_t^\ast$ containing $v$ is more than $\sqrt{\zeta}$.
			We now bound the number of special, unavailable and dangerous vertices.
			
			First observe that, since $B_i^\ast$ has order at least $(s-1)(1-\gamma)m_{1.5}+1 \geq sm_{1.5}/2$ for each $i$, we have $j \le t\leq n/(sm_{1.5}/2)$.
			It follows that the number of unavailable vertices is at most $t sm_1 \leq \zeta n$.
			\fw{Recalling in addition that each part of the $B_i^\ast$ has size at most $2(1+\gamma)m_{1.5}$, we have that the number of special vertices is at most $2s  \cdot 2(1+\gamma)m_{1.5} \le \zeta n$, where we used that $1/n \ll 1/m_{1.5} \ll \zeta,1/s$.}
			Moreover, with $P$ being uniformly chosen among all parts of the blow-ups $B_i^\ast$, we have
			\begin{equation*}
				\mathbb{E}_P \biggl[\frac{\abs{P\cap \bigcup_{j'\leq j-1} V(C_{j'}^\ast(m_1))}}{\abs{P}}\biggr] \leq \zeta\, .
			\end{equation*}
			Note that $\frac{\abs{P\cap \bigcup_{j'\leq j-1} V(C_{j'}^\ast(m_1))}}{\abs{P}}$ is the fraction of unavailable vertices contained in the part $P$.
			Therefore, by Markov's inequality, at most $\sqrt{\zeta}$-fraction of all parts (across all the blow-ups~$B_i^\ast$) contain more than $\sqrt{\zeta}$-fraction of unavailable vertices. Since  there are at most $\frac{n}{(1-\gamma)m_{1.5}}$ parts and each part has size at most $2(1+\gamma)m_{1.5}$, the number of dangerous vertices is at most $\sqrt{\zeta} \cdot \frac{n}{(1-\gamma)m_{1.5}} \cdot (2(1+\gamma)m_{1.5}) \le 3\sqrt{\zeta}n$.
			\fw{In particular, there are at most $ \zeta n + \zeta n+3\sqrt{\zeta}n\le  4\sqrt{\zeta} n$ vertices which are special, unavailable or dangerous, where we used $\zeta \ll 1/s$. We conclude that no more than $4\sqrt{\zeta} n \cdot n^{s-2k - 1} = 4\sqrt{\zeta} n^{s-2k}$ distinct $(s-2k)$-subsets of $V(H)$ contain a vertex which is special, unavailable or dangerous.}
			
			Let $T$ be any subset of vertices of $H$ of size $2k$ obtained by picking one vertex from each part of $b^{j-1}_{\text{right}}$ and $b^{j}_{\text{left}}$.
			Then \cref{lem:robustmoregeneral} and the calculation above imply that the number of  $(s-2k)$-subsets $S$ of $V(H) \setminus T$, not containing any special, unavailable or dangerous vertex and such that $H[T\cup S]\in \mathcal{P}(\eps/2,k)$, is at least
			\begin{equation}
				\label{eq:lemma3.3_applicable}
				\fw{(1-e^{-\sqrt{s}}) \binom{\abs{V(H) \setminus T}}{s-2k} - 4\sqrt{\zeta} n^{s-2k} \ge n^{s-2k}/s!\, .}
			\end{equation}
			
			Let $G$ be the $k$-subgraph of $H$ obtained after removing all special, unavailable and dangerous vertices.
			\fw{We apply \cref{lem:rooted-blow-ups} to $G$ with $m_{1.5}, m_1, 2k, 1/s!$ playing the roles of $m_2,m_1, q, \mu$, with $\mathcal{V}$ being the collection of the parts in $b^{j-1}_{\text{left}}$ and $b^{j}_{\text{right}}$, and with $\cP=\cP(\eps/2,k)$.} Observe that the hypotheses of \cref{lem:rooted-blow-ups} are satisfied by inequality~\eqref{eq:lemma3.3_applicable}.
			Therefore there is a family $\cU$ of $s-2k$ pairwise disjoint $m_1$-sets and disjoint with $\cV$ such that the $s$-partite graph induced by $\cV \cup \cU$ contains a consistent $m_1$-regular blow-up of some $P \in \cP(\eps/2,k)$.
			Denote this blow-up by $C_j^\ast$ and observe it satisfies the condition in the statement of the claim.
			\fw{Moreover, it satisfies \ref{item:claim_dangerous} by the following argument.}
			The blow-up~$C_j^\ast$ uses $(s-2k) m_1$ new vertices and all such vertices are not dangerous.
			Given a part $B$ of the blow-up $B_i^\ast$, recall that $(1-\gamma)m_{1.5} \le \abs{B} \le 2(1+\gamma)m_{1.5}$.
			Then the number of vertices of such part used by the blow-ups $C_1^\ast, \dots, C_{j}^\ast$ is at most $\sqrt{\zeta}\abs{B}+(s-2k)m_1 \le \gamma \abs{B}$, using $\zeta \ll \gamma$ and $1/m_{1.5} \ll 1/m_1, 1/s, \gamma$.
			
			The second part of the claim follows directly from \ref{item:claim_dangerous} (for $j=t$).
		\end{proof}
		
		We now redefine $B_i^\ast$ as its subset obtained by removing all vertices in $\bigcup_{i \in [t-1]} C_i^\ast$ with the exception of the vertices contained in the $k$-edges guaranteed by \cref{claim:glue}. Then, the sequence of blow-ups
		\begin{equation*}
			B_1^\ast, C_1^\ast, B_2^\ast,C_2^\ast, B_3^\ast,\dotsc, C_{t-1}^\ast, B^\ast_{t}
		\end{equation*}
		with their underlying graphs
		\begin{equation*}
			B_1, C_1, B_2, C_2, B_3, \dotsc, C_{t-1}, B_t
		\end{equation*}
		satisfies the properties \ref{blowup_1} through \ref{blowup_5} in place of $F_1^\ast, \dotsc, F^\ast_\ell$ and $F_1, \dotsc, F_\ell$ respectively (and with $2\gamma$ in place of $\gamma$).
	\end{proof}

	\section{Geometric observations}
	\label{sec:geometric_obs}
	
	Recall that a \defn{simplicial $k$-sphere} is a homogeneous simplicial $k$-complex that is homeomorphic to $\bS^k$. In this section we will first introduce some ways to make new simplicial spheres from old ones and then show that certain $k$-graphs contain spanning copies of $\bS^{k - 1}$ (see \cref{lem:partitesphere,lem:tightpathsphere}).
	
	The first operation that makes new simplicial spheres is gluing along a facet.
	
	\begin{remark}\label{rmk:gluecommonface}
		Let $\cK$ and $\cK'$ both be simplicial $k$-spheres whose intersection is a $k$-simplex $F$. Let $\cK''$ be the simplicial complex whose vertex-set is $V(\cK) \cup V(\cK')$ and whose simplices are those of $\cK$ and $\cK'$ except $F$. This operation glues $\cK$ and $\cK'$ on $F$ and so $\cK''$ is a simplicial $k$-sphere.
	\end{remark}
	
	Given a topological space $X$, the \defn{suspension} of $X$ is obtained by taking the cylinder $X \times [0, 1]$, contracting $X \times \set{0}$ to a single point, and contracting $X \times \set{1}$ to a single point. Crucially, the suspension of $\bS^{k - 1}$ is $\bS^k$ (for $k \geq 1$). We now define the analogous operation for simplicial complexes.
	
	\begin{definition}[Suspension]
		Let $\cK$ be a homogeneous simplicial $(k - 1)$-complex. The \defn{suspension} $\cK'$ of $\cK$ is a homogeneous simplicial $k$-complex whose vertex-set is $\set{u, v} \cup V(\cK)$, where $u$ and $v$ are new vertices, and whose set of simplices is
		\begin{equation*}
			\set{\set{u} \cup F \colon F \in E(\cK)} \cup \set{\set{v} \cup F \colon F \in E(\cK)}.
		\end{equation*}
	\end{definition}
	
	For example, the suspension of a 4-cycle is an octahedron.
	
	\begin{lemma}\label{lem:suspensions}
		The suspension of a simplicial $(k - 1)$-sphere is a simplicial $k$-sphere.
	\end{lemma}
	
	\begin{proof}
		Let $\cK$ be a simplicial $(k - 1)$-sphere. Embed $\cK$ in $\bR^k$ and consider it to be in the codimension one flat $x_1 = 0$ in $\bR^{k + 1}$. Add two vertices $u$ and $v$ at $(1, 0, \dotsc, 0)$ and $(-1, 0, \dotsc, 0)$, respectively. For each simplex $F$ of $\cK$, its suspension $\cK'$ has the simplices $F \cup \set{u}$ and $F \cup \set{v}$. The topological space induced by the simplices of $\cK'$ is the suspension of the topological space induced by the simplices of $\cK$ and so is $\bS^k$.
	\end{proof}
	
	We now show that certain $k$-partite $k$-graphs contain spanning copies of $\bS^{k - 1}$.
	Given positive integers $k$ and $a_1, \dotsc, a_k \ge 1$, we denote by \defn{$K_k^{(k)}(a_1, \dotsc, a_k)$} the complete $k$-partite $k$-graph with parts of size $a_1, \dotsc, a_k$.
	
	\begin{lemma}\label{lem:partitesphere}
		For any $k \geq 2$, the following $k$-partite $k$-graphs contain spanning copies of $\bS^{k - 1}$\textup{:}
		\begin{enumerate}[label = \textup{(}\alph*\textup{)}]
			\item \label{itm:partitesphere-a} $K_k^{(k)}(2, \dotsc, 2, 2, \ell, \ell)$ for any $\ell \geq 2$,
			\item \label{itm:partitesphere-b} $K_k^{(k)}(2, \dotsc, 2, 3, \ell, \ell)$ for any $\ell \geq 3$.
		\end{enumerate}
	\end{lemma}
	
	\begin{proof}
		Note that the suspension of $K_k^{(k)}(a_1, \dotsc, a_k)$ is $K_{k + 1}^{(k + 1)}(2, a_1, \dotsc, a_k)$. Hence, if we show that $K_2^{(2)}(\ell, \ell)$ (where $\ell \geq 2$) contains a spanning copy of $\bS^1$ and $K_3^{(3)}(3, \ell, \ell)$ (where $\ell \geq 3$) contains a spanning copy of $\bS^2$, the statement of the lemma follows by induction (using \cref{lem:suspensions}).
		
		Let $\ell \geq 2$. Observe that $K_2^{(2)}(\ell, \ell)$ contains a spanning copy of $C_{2 \ell}$ which is homeomorphic to $\bS^1$, as desired for part~\ref{itm:partitesphere-a}.
		
		Now, let $\ell \geq 3$. $K_3^{(3)}(2, \ell - 1, \ell - 1)$ contains a spanning copy of a $\bS^2$ by part~\ref{itm:partitesphere-a} (alternatively, note that $K_3^{(3)}(2, \ell - 1, \ell - 1)$ contains a spanning $(2\ell - 2)$-gonal bipyramid). Fix a particular facet $u_1 u_2 u_3$ of this spanning copy of $\bS^2$ where $u_1$ is in the first part of $K_3^{(3)}(2, \ell - 1, \ell - 1)$,~$u_2$ is in the second part, and $u_3$ is in the third. Subdivide the facet $u_1 u_2 u_3$ into facets $v_1 u_2 u_3$, $u_1 v_2 u_3$, $u_1 u_2 v_3$, $u_1 v_2 v_3$, $v_1 u_2 v_3$, $v_1 v_2 u_3$, and $v_1 v_2 v_3$ as shown in \cref{fig:triangulation}.
		
		\begin{figure}[H]
			\centering
			\begin{subfigure}[m]{0.3\textwidth}
				\centering
				\begin{tikzpicture}
					\foreach \pt in {1, 2, 3}{
						\tkzDefPoint(120*\pt - 30:2){u_\pt}
						\tkzDrawPoint(u_\pt)
					}
					\tkzDrawPolygon(u_1,u_2,u_3)
					\tkzLabelPoint[above](u_1){$u_1$}
					\tkzLabelPoint[below left](u_2){$u_2$}
					\tkzLabelPoint[below right](u_3){$u_3$}
				\end{tikzpicture}
			\end{subfigure}%
			$\rightarrow$
			\begin{subfigure}[m]{0.3\textwidth}
				\centering
				\begin{tikzpicture}
					\foreach \pt in {1, 2, 3}{
						\tkzDefPoint(120*\pt - 30:2){u_\pt}
						\tkzDrawPoint(u_\pt)
						\tkzDefPoint(120*\pt +150:0.4){v_\pt}
						\tkzDrawPoint(v_\pt)
					}
					\tkzDrawPolygon(u_1,u_2,u_3)
					\tkzDrawPolygon(v_1,v_2,v_3)
					\tkzDrawSegments(u_1,v_2 u_1,v_3 u_2,v_1 u_2,v_3 u_3,v_1 u_3,v_2)
					\tkzLabelPoint[above](u_1){$u_1$}
					\tkzLabelPoint[below left](u_2){$u_2$}
					\tkzLabelPoint[below right](u_3){$u_3$}
					\tkzLabelPoint[below](v_1){$v_1$}
					\tkzLabelPoint[above right=-2pt](v_2){$v_2$}
					\tkzLabelPoint[above left=-2pt](v_3){$v_3$}
				\end{tikzpicture}
			\end{subfigure}
			\caption{}\label{fig:triangulation}
		\end{figure}
		
		Putting $v_i$ into the $i$th part for each $i$ shows that $K_3^{(3)}(3, \ell, \ell)$ contains a spanning copy of~$\bS^2$, as required for part~\ref{itm:partitesphere-b}.
	\end{proof}
	
	We now show that certain blow-ups of tight paths contain large copies of $\bS^{k - 1}$ with special properties. We say that a copy $S$ of $\bS^{k - 1}$ in a blow-up of a tight path $P$ is \defn{doubly edge-covering} if there are two families $\set{f_e \colon e \in E(P)}$, $\set{f'_e \colon e \in E(P)}$ of facets of $S$ such that each family is vertex-disjoint, and, for each $e \in E(P)$, $f_e \neq f'_e$ and $\phi(f_e) = \phi(f'_e) = e$. For positive integers $a_1, \dotsc, a_\ell$ (where $\ell \geq k \geq 2$), \defn{$P_\ell^{(k)}(a_1, \dotsc, a_\ell)$} denotes a blow-up of the $k$-uniform tight path on $\ell$ vertices where the $i$th vertex has been blown-up by $a_i$.
	
	\begin{lemma}\label{lem:thin-path}
		For $k\geq 2$, the blow-up $P_{k + 1}^{(k)}(1, 2, 2, \dotsc, 2, 2, 1)$ contains a spanning copy of $\bS^{k - 1}$ that is doubly edge-covering.
	\end{lemma}
	\begin{proof}
		For $k = 2$, $P_{3}^{(2)}(1, 2, 1)$ is a 4-cycle $u_0 u_1 u_2 v_1$ where $u_i$, $v_i$ are in the $i$th part. The two families $\set{u_0 u_1, v_1 u_2}$ and $\set{u_0 v_1, u_1 u_2}$ show that the 4-cycle (which is homeomorphic to $\bS^1$) is doubly edge-covering. Now let $k \geq 3$ and consider $P_{k + 1}^{(k)}(1, 2, 2, \dotsc, 2, 2, 1)$ which has vertices $u_0, u_1, \dotsc, u_k, v_1, \dotsc, v_{k - 1}$ and underlying tight path $b_0 \dotsc b_k$ where $\phi(u_i) = \phi(v_i) = b_i$ for all $i$. This $P_{k + 1}^{(k)}(1, 2, 2, \dotsc, 2, 2, 1)$ is the suspension of the $K_{k - 1}^{(k - 1)}(2, 2, \dotsc, 2)$ that has parts $(\set{u_i, v_i} \colon 1 \leq i \leq k - 1)$. By \cref{lem:partitesphere}, this $K_{k - 1}^{(k - 1)}(2, 2, \dotsc, 2)$ contains a spanning copy $S'$ of $\bS^{k - 2}$. Furthermore, by considering the proof of \cref{lem:partitesphere}, we may assume that $u_1 \dotsc u_{k - 1}$ and $v_1 \dotsc v_{k - 1}$ are both facets of $S'$. Thus, relying on \cref{lem:suspensions}, $P_{k + 1}^{(k)}(1, 2, \dotsc, 2, 2, 1)$ contains a spanning copy $S$ of $\bS^{k - 1}$ whose facets include $u_0 u_1 \dotsc u_{k - 1}$, $u_0 v_1 \dotsc v_{k - 1}$, $u_1 \dotsc u_{k - 1} u_k$, $v_1 \dotsc v_{k - 1} u_k$ ($S$ is the suspension of $S'$). The underlying tight path only has two edges, $b_0 \dotsc b_{k - 1}$ and $b_1 \dotsc b_k$, and the two families $\set{u_0 u_1 \dotsc u_{k - 1}, v_1 \dotsc v_{k - 1} u_k}$ and $\set{u_0 v_1 \dotsc v_{k - 1}, u_1 \dotsc u_{k - 1} u_k}$ show that $S$ is doubly edge-covering.
	\end{proof}
	
	\begin{lemma}\label{lem:growing-path}
		Let $\ell - 1 \geq k \geq 2$.
		If the blow-up  $P_\ell^{(k)}(a_1, \dotsc, a_\ell)$ contains a spanning copy of~$\bS^{k - 1}$ that is doubly edge-covering, then so does $P_{\ell + 1}^{(k)}(1, a_1 + 1, \dotsc, a_{k - 1} + 1, a_k, \dotsc, a_\ell)$.
	\end{lemma}
	\begin{proof}
		Suppose that $P_\ell^{(k)}(a_1, \dotsc, a_\ell)$ contains a doubly edge-covering spanning copy $S$ of~$\bS^{k - 1}$ with corresponding families of vertex-disjoint facets $\set{f_e \colon e \in E(P)}$, $\set{f'_e \colon e \in E(P)}$ where $P$ is the underlying tight path. By \cref{lem:thin-path}, $P_{k + 1}^{(k)}(1, 2, 2, \dotsc, 2, 1)$ contains a doubly edge-covering spanning copy $T$ of $\bS^{k - 1}$ with corresponding families of vertex-disjoint facets $\set{g_e \colon e \in E(Q)}$, $\set{g'_e \colon e \in E(Q)}$ where $Q$ is the underlying tight path.
		Identify $P$ as $b_1 \dotsc b_\ell$, $Q$ as $b_0 \dotsc b_k$, and identify facet $f'_{b_1 \dotsc b_k}$ with facet $g_{b_1 \dotsc b_k}$ (but keep all other vertices of $S$ and $T$ distinct). Glue $S$ and $T$ along the common facet $f'_{b_1 \dotsc b_k} = g_{b_1 \dotsc b_k}$. By \cref{rmk:gluecommonface}, the resulting simplicial complex $U$ is homeomorphic to $\bS^{k - 1}$. Furthermore, $U$ spans $P_{\ell + 1}^{(k)}(1, a_1 + 1, \dotsc, a_k + 1, a_{k + 1}, \dotsc, a_\ell)$ with underlying tight path $R = b_0 b_1 \dotsc b_\ell$.
		
		Finally, we claim that the two families $\cF = \set{f_e \colon e \in E(P)} \cup \set{g_e \colon e \in E(Q), e \neq b_1 \dotsc b_k}$ and $\cF' = \set{f'_e \colon e \in E(P), e \neq b_1 \dotsc b_k} \cup \set{g'_e \colon e \in E(Q)}$ witness that $U$ is doubly edge-covering. Each edge of $R$ is an edge of $P$ or an edge of $Q$ or an edge of both (if the edge is $b_1 \dotsc b_k$) and so, for every edge of $e \in E(R)$, there is exactly one facet $h_e \in \cF$ and exactly one facet $h'_e \in \cF'$ with $\phi(h_e) = \phi(h'_e) = e$. Furthermore, $h_e \neq h'_e$ since $f_e \neq f'_e$ for all $e \in E(P)$ and $g_e \neq g'_e$ for all $e \in E(Q)$. We now check that the family $\cF$ is vertex-disjoint. Each of the families $\set{f_e \colon e \in E(P)}$ and $\set{g_e \colon e \in E(Q), e \neq b_1 \dotsc b_k}$ is vertex-disjoint. By construction, $V(S) \cap V(T) = V(g_{b_1 \dotsc b_k})$ and so the only vertices that can be in both $f_e, g_{e'} \in \cF$ are those in $g_{b_1 \dotsc b_k}$. However $\set{g_e \colon e \in E(Q)}$ is a vertex-disjoint family and so every facet in $\set{g_e \colon e \in E(Q), e \neq b_1 \dotsc b_k}$ is vertex-disjoint from $g_{b_1 \dotsc b_k}$ and so is vertex-disjoint from every facet in $\set{f_e \colon e \in E(P)}$. Thus $\cF$ is vertex-disjoint. Similarly, $\cF'$ is vertex-disjoint, as required.
	\end{proof}
	
	\begin{lemma}\label{lem:tightpathsphere}
		For $\ell - 1 \geq k \geq 2$, the blow-up $P_\ell^{(k)}(k, \dotsc, k)$ contains a \textup{(}not necessarily spanning\textup{)} copy $S$ of $\bS^{k - 1}$ that is doubly edge-covering.
	\end{lemma}
	
	\begin{proof}
		Repeatedly applying \cref{lem:thin-path,lem:growing-path} shows that each of the following contains a spanning copy of $\bS^{k - 1}$ that is doubly edge-covering:
		\begin{align*}
			& P_{k + 1}^{(k)}(1, \underbrace{2, 2, \dotsc, 2, 2}_{k - 1}, 1),                                          \\
			& P_{k + 2}^{(k)}(1, 2, \underbrace{3, 3, \dotsc, 3, 3}_{k - 2}, 2, 1),                                    \\
			& P_{k + 3}^{(k)}(1, 2, 3, \underbrace{4, 4, \dotsc, 4, 4}_{k - 3}, 3, 2, 1),                              \\
			& \qquad \qquad \qquad \qquad \vdots                                                                       \\
			& P_{2k - 1}^{(k)}(1, 2, \dotsc, k - 1, k, k - 1, \dotsc, 2, 1),                                           \\
			& \qquad \qquad \qquad \qquad \vdots                                                                       \\
			& P_\ell^{(k)}(1, 2, \dotsc, k - 1, \underbrace{k, k, \dotsc, k, k}_{\ell - 2k + 2}, k - 1, \dotsc, 2, 1).
		\end{align*}
		Thus $P_\ell^{(k)}(k, \dotsc, k)$ contains a (not necessarily spanning) copy of $\bS^{k - 1}$ that is doubly edge-covering.
	\end{proof}

	\section{Allocation}
	\label{sec:allocation}
	
	This section is dedicated to the proof of \cref{lem:allocation}.
	
	\subsection{Filling the blow-ups with spheres}
	
	The following result allows us to cover the vertices in a suitable blow-up with pairwise vertex-disjoint simplicial spheres.
	For a $k$-graph $R$, we define \defn{$\partial_2 R$} to be the set of pairs of vertices of $R$ which are contained in at least one edge of $R$.
	
	\begin{lemma}\label{lem:filling}
		Let $1/m \ll 1/s \ll \gamma \ll \eps, 1/k \le 1/3$ with $k$, $s$, and $m$ positive integers. Let $R$ be an $s$-vertex $k$-graph without isolated vertices and with $\comin(R) \ge (1/2 + \eps)s$.
		Let $R^\ast$ be {$(\gamma, m)$-regular blow-up} of $R$.
		Let $\set{f_e \colon e \in E(R)}$ be a collection of pairwise vertex-disjoint edges of $R^\ast$ with $\phi(f_e) = e$ for all $e \in E(R)$.
		
		Then there exists a family of copies of $\bS^{k - 1}$, $\set{S_e \colon e \in E(R)}$, such that $f_e$ is a facet of $S_e$ for each $e \in E(R)$ and $\set{V(S_e) \colon e \in E(R)}$ is a partition of $V(R^\ast)$.
	\end{lemma}
	
	\begin{proof}
		We begin by assigning to each pair $p \in \partial_2 R$ an edge $e_p \in E(R)$ such that $p \subs e_p$ for each $p \in \partial_2 R$ and the $e_p$ are pairwise distinct.
		This is possible by the following argument.
		Define an auxiliary bipartite graph $G_R$ with parts $\partial_2 R$ and $E(R)$ where $p \in \partial_2 R$ is joined to $e \in E(R)$ if and only if $p \subs e$. It suffices to find a matching between $\partial_2 R$ and $E(R)$ covering $\partial_2 R$. Consider some $p \in \partial_2 R$. There is at least one edge containing $p$ and so, since $R$ has $\comin(R) \ge (1/2+\eps)s$, $p$ is contained in at least $(1/2 + \eps)s$ edges. In the other direction, every $e \in E(R)$ contains exactly $\binom{k}{2}$ pairs $p \in \partial_2 R$. Thus, in $G_R$, every vertex of $\partial_2 R$ has degree at least $s/2$ while each vertex of $E(R)$ has degree $\binom{k}{2}$. Now $1/s \ll 1/k$ so every vertex in $\partial_2 R$ has degree greater than every vertex in $E(R)$. Thus, by Hall's marriage theorem, there is a matching between $\partial_2 R$ and $E(R)$ covering $\partial_2 R$.
		This gives the required assignment $(e_p \colon p \in \partial_2 R)$.
		
		We now move our attention to $R^\ast$.
		Let us define a family $\{A_e \colon e \in E(R)\}$ of pairwise disjoint $2k$-sets such that for each edge $e= u_1 \dotsc u_k$, we have that $f_e \subseteq A_e$ and $A_e$ has exactly two vertices in each $\phi^{-1}(u_i)$.
		This is possible since the edges $f_e$ are pairwise disjoint, $f_e \subset \phi^{-1}(e)$ for each $e$, and the number of vertices in the union of the $A_e$ is at most $2k \cdot \abs{E(R)} \leq k s^k$, which is much less than the number of vertices in any part.
		
		For an upcoming matching argument, we require that the number of vertices in $R^\ast - \bigcup_e A_e$ is even.
		If this this is not the case, then we alter exactly one of the $A_e$ as follows. Pick some $e^\ast \in E(R)$ that is not in $(e_p \colon p \in \partial_2 R)$ which is possible as $\abs{E(R)}$ is much larger than $\abs{\partial_2 R}$. Write $e^\ast = u^\ast_1 \dotsc u^\ast_k$. Add to $A_{e^\ast}$ one vertex from each of $\phi^{-1}(u_1^\ast)$, $\phi^{-1}(u_2^\ast)$, $\phi^{-1}(u_3^\ast)$ so that the $A_e$ are still pairwise disjoint. Now the number of vertices in $R^\ast - \bigcup_e A_e$ is even.
		
		We claim that the graph $F$ that is induced by $\partial_2 (R^\ast - \bigcup_e A_e)$ has a perfect matching.
		Recall that the order of $F$ is even.
		So by Dirac's theorem, it suffices to show that $F$ has minimum degree at least $\abs{V(F)}/2$.
		Since $R$ has no isolated vertices, each vertex $w \in V(R)$ is contained in at least one edge and thus in a supported $(k-1)$-set.
		Since $\comin(R) \ge (1/2 + \eps)s$, all the $(1/2+\epsilon)s$ other vertices that make an edge with that set are neighbours of $w$ in $\partial_2 R$.
		Therefore $\partial_2 R$ has minimum degree at least $(1/2 +\eps)s$.
		Consider a vertex $v \in V(F)$, and write $x = \phi(v)$.
		Denote the parts of the blow-up $R^\ast$ by $\{U_y\colon y \in V(R)\}$.
		Note that $v$ is adjacent (in $F$) to all vertices in every part $U_y$ with $xy \in E(\partial_2 R)$, except for the elements of the sets $A_e$.
		Since $R^\ast$ is a $(\gamma, m)$-regular blow-up, it follows that
		\begin{align*}
			\delta_{F}(v) & \geq \sum_{y\colon xy \in \partial_2 R}(1/2 +\eps)s \abs{U_x} - \sum_{e\in E(R)} \abs{A_e} \\
			& \geq(1/2 +\eps)s (1-\gamma)m - 3k\tbinom{s}{2}                                             \\
			& \geq s(1+\gamma)m/2                                                                        \\
			& \geq \abs{V(F)}/2,
		\end{align*}
		where the last inequality follows from the choices of the constants.
		In conclusion, $F$ has a perfect matching $M$.
		
		We now add vertices to the $A_e$ to form sets $B_e$ ($e \in E(R)$). Firstly, if $e \not\in \set{e_p \colon p \in \partial_2 R}$, then $B_e \coloneqq A_e$. In particular, $e^\ast$ (if it exists) satisfies $B_{e^\ast} = A_{e^\ast}$. For $e = e_p$ ($p = ab \in \partial_2 R$), let $M_p$ be the sets of edges of the matching $M$ that go between $\phi^{-1}(a)$ and $\phi^{-1}(b)$ and set $B_e \coloneqq A_e \cup V(M_p)$. Observe that, since $M$ is a perfect matching, the sets $B_e$ partition $V(R^\ast)$. Moreover, since $p \subs e_p$ for each $p \in \partial_2 R$, every $B_e$ is a subset of $\phi^{-1}(e)$ and so $R^\ast[B_e]$ is a complete $k$-partite $k$-graph. Counting vertices in each part shows
		\begin{equation*}
			R^\ast[B_e] \cong
			\begin{cases}
				K_k^{(k)}(2,  \dotsc, 2, 3, 3, 3)                               & \text{if } e = e^\ast,                                \\
				K_k^{(k)}(2,  \dotsc, 2, \ell, \ell) \text{ where } \ell \geq 2 & \text{if } e \in \set{e_p \colon p \in \partial_2 R}, \\
				K_k^{(k)}(2,  \dotsc, 2)                                        & \text{otherwise}.
			\end{cases}
		\end{equation*}
		By \cref{lem:partitesphere}, $R^\ast[B_e]$ contains a spanning copy $S_e$ of $\bS^{k - 1}$, for each $e \in E(R)$. Furthermore, the edge $f_e$ is a subset of $A_e \subset B_e$ and so, by symmetry, we may choose $S_e$ such that $f_e$ is a facet of $S_e$. Finally, since the $B_e$ form a partition of $V(R^\ast)$, we have that $\set{V(S_e) \colon e \in E(R)}$ is a partition of $V(R^\ast)$.
	\end{proof}

	\subsection{Connecting spheres}
	
	Now we are ready to finish the proof of \cref{lem:allocation}.
	
	\begin{proof}[Proof of \cref{lem:allocation}]
		Fix constants satisfying $1/m \ll 1/s \ll \gamma \ll \eps,1/k \le 1/3$ such that they also fulfil the requirements of \cref{lem:filling} with $2\gamma$ playing the role of $\gamma$.
		Let $R$ be an $s$-vertex $k$-graph without isolated vertices and with $\comin(R) \ge (1/2+\eps)s$.
		Let $R^\ast$ be a $(\gamma, m)$-nearly-regular blow-up of $R$ and let $f_1,f_2\in E(R^\ast)$ such that $\phi(f_1)$, $\phi(f_2)$ and the part of $R^\ast$ which is a singleton (if it exists) are all disjoint.
		
		If $R^\ast$ contains a singleton, say $v$, then let $e$ be an edge of $R$ that contains $\phi(v)$ and is disjoint from $\phi(f_1)$ and $\phi(f_2)$.
		Further let $u$ be a vertex from $R^\ast$ such that $e'=(e \cup \set{ \phi(u) }) \setminus \set{\phi(v) }$ is an edge of $R$ and $\phi(u)$ is not contained in $\phi(f_1)$ and $\phi(f_2)$.
		Note that these choices are possible as $\comin(R) \geq (1/2 + \eps) s$.
		For each of the $k - 1$ vertices in {$e \setminus \set{\phi(v) }$}, consider its corresponding part in $R^\ast$ and take any two of its vertices.
		The union of these $2(k - 1)$ vertices, together with the vertices $u$ and $v$, gives a copy of $K_{k}^{(k)}(2,\dots,2)$ that contains a spanning copy $S'$ of $\bS^{k - 1}$ by \cref{lem:partitesphere}.
		By symmetry we can assume that there is a facet $f_3$ that does not contain $v$ and note that $\phi(f_3) = e'$ is disjoint from $\phi(f_1)$ and $\phi(f_2)$ by construction.
		We remove the vertices of $S'$ from $R^\ast$, except for those in $f_3$, and we remove $\phi(v)$ from $R$.
		Note that we still have $\comin(R) \geq (1/2 + \eps/4) (s-1)$ and that $R$ has no isolated vertices.
		If~$R^\ast$ does not contain a singleton, let $f_3$ be an arbitrary edge of $R^\ast$ such that $\phi(f_3)$ is disjoint from~$\phi(f_1)$ and $\phi(f_2)$.
		
		By \cref{lem:dirac-to-tightly-connected}, $R$ is tightly connected and thus, by \cref{lem:tight-connectivity-bounded-tight-walk}, we can find a tight walk $W$ of order $t$ with $t \le s^{2k}$, which contains each edge of $R$ as a subwalk.
		This implies that there is a tight path $P$ in $R^\ast$ such that for every $e \in E(R)$ there exists $f \in E(P)$ with $\phi(f) = e$.
		Indeed $P$ can be obtained as follows.
		Let $w_1, w_2, \dotsc, w_t$ be the cyclic order of the vertices giving $W$ and, for each $i \in [t]$, choose a  vertex $u_i \in V(R^\ast)$ such that $\phi(u_i)=w_i$ and all the $u_i$ are pairwise disjoint.
		Then $P=u_1 u_2 \dotsc u_t$ is the desired tight path in $R^\ast$.
		Note that this is possible as the size of each part of $R^\ast$ is at least $(1-\gamma)m$ which is much greater than the order of $W$.
		
		In particular, since $R^\ast$ is a blow-up graph with each part of size at least $(1-\gamma)m$ and $1/m \ll 1/s, \gamma, 1/k$, it follows that $R^\ast$ contains the $(k,0)$-regular blow-up of $P$, which is isomorphic to $P_{t}^{(k)}(k,\dots,k)$.
		Then, by \cref{lem:tightpathsphere} (and since $t \ge k+1$), $R^\ast$ contains a copy $S$ of $\bS^{k - 1}$ that is doubly edge-covering. Let the two families of facets of $S$ witnessing this be $\set{f_e \colon e \in E(R)}$ and $\set{f'_e \colon e \in E(R)}$.
		Without loss of generality, we can assume that $f_1 = f_{\phi(f_1)}'$, $f_2 = f_{\phi(f_2)}'$ and $f_3 = f_{\phi(f_3)}'$ are facets of $S$.
		Note that $S$ uses at most $k \cdot 2s^{k} \le \gamma m$ vertices from each part of $R^\ast$.
		
		We let $\tilde R$ be the blow-up of $R$ obtained from $R^\ast$ by removing $V(S) \setminus \bigcup_{e \in E(R)} V(f_e)$.
		Each part of $\tilde{R}$ has still size in the interval $(1 \pm 2\gamma)m$, so we can apply \cref{lem:filling} with $\set{ f_e \colon e\in E(R) }$ to cover the vertices of $\tilde R$ with copies of $\bS^{k - 1}$, $\set{S_e \colon e \in E(R) }$, such that $f_e$ is a facet of $S_e$ for each $e$ and $\set{V(S_e) \colon e \in E(R)}$ is a partition of $V(\tilde{R})$.
		Finally, by \cref{rmk:gluecommonface}, we can, for each $e \in E(R)$, glue $S_e$ to $S$ on facet $f_e$ and glue $S'$ to $S$ on facet $f_3$ to obtain a spanning copy of $\bS^{k - 1}$ in $R^\ast$, where $f_1$ and $f_2$ are facets.
	\end{proof}
	
	{
		\fontsize{11pt}{12pt}
		\selectfont
		
		\hypersetup{linkcolor={red!70!black}}
		\setlength{\parskip}{2pt plus 0.3ex minus 0.3ex}
		
		\newcommand{\etalchar}[1]{$^{#1}$}
		
	}
	
	\appendix
	
	\section{Covering most vertices with blow-ups}\label{lem:almost-blow-up-cover}
	
	In this section, we show \cref{lem:covering-with-blow-ups} following the original exposition~\cite[Lemma~4.4]{lang2023tiling}.
	We require a simple Dirac-type bound for hypergraph matchings due to Daykin and H\"{a}ggkvist~\cite{DH81} (for a short proof, see~\cite[Lemma~B.2]{langsanhueza2024dirac}).
	
	\begin{lemma}\label{lem:matching}
		Let $ 1/n \ll 1/s, \mu$ with $n$ divisible by $s$.
		Let $P$ be an $n$-vertex $s$-graph with $\delta_1(P) \geq (1 - 1/s + \mu) \tbinom{n-1}{s-1}$.
		Then $P$ has a perfect matching.
	\end{lemma}
	
	The proof of \cref{lem:covering-with-blow-ups} is then based on the next fact, which tiles a $k$-graph $P$ almost perfectly with constant-sized blow-ups of edges, provided that the minimum degree of $P$ forces a perfect matching (as in \cref{lem:matching}).
	For a $k$-graph $H$ and a family of $k$-graphs $\cR$, an \defn{$\cR$-tiling} is a set of pairwise vertex-disjoint $k$-graphs $R_1,\dots,R_\ell \subset H$ with $R_1,\dots,R_\ell \in \cR$.
	
	\begin{lemma}\label{lem:blow-up-matching}
		Let $1/n \ll \mu,\, 1/b,\, 1/s$.
		Then every $n$-vertex $s$-graph $P$ with $\delta_1(P) \geq (1-1/s+\mu) \binom{n-1}{s-1}$ contains a $K_s^{(s)}(b)$-tiling missing at most $\mu n$ vertices.
	\end{lemma}
	
	To derive \cref{lem:covering-with-blow-ups}, we first apply \cref{lem:blow-up-matching} to cover most vertices of the property $s$-graph with complete partite graphs.
	Then we use \cref{thm:erd64} to distil most vertices of each of the partite graphs into the desired blow-ups.
	
	\begin{proof}[Proof of \cref{lem:covering-with-blow-ups}]
		Introduce $m'$ such that $ 1/n \ll 1/m' \ll 1/k,\,1/s,\, 1/m ,\, \eta$.
		We apply \cref{lem:blow-up-matching} to the $n$-vertex $s$-graph $\PG{H}{ \cP }{s}$ to obtain a $K_s^{(s)}(m')$-tiling missing at most $(\eta/2) n$ vertices.
		Fix a copy $K$ of $K_s^{(s)}(m')$ in this tiling.
		Note that each edge of $K$ corresponds to an $s$-vertex $k$-graph satisfying $\cP$.
		To finish, we have to cover all but at most $(\eta/2) m's$ vertices of $K$ with blow-ups $R^\ast_1(m),\dots,R^\ast_{\ell'}(m)$, where each $R_i$ is an $s$-vertex $k$-graph satisfying $R_i \in \cP$.
		
		We colour each edge $Y$ of $K$ by one of at most $2^{s^{k}}$ colours corresponding to the $k$-graph~$H[Y]$ labelled by the part indices.
		Applying \cref{thm:erd64} to the colour with most edges, we identify a complete $s$-partite subgraph $K' \subset K$ with parts of size $m$ whose edges all correspond to the same $s$-vertex $k$-graph~$R$ with $R \in \cP$.
		In other words, $K'$ is isomorphic to the blow-up~$R^\ast(m)$.
		We take out the vertices of $K'$ and repeat this process until all but at most $(\eta/2) m's$ vertices of~$K$ are covered, as desired.
	\end{proof}
	
	It remains to show \cref{lem:blow-up-matching}.
	A first proof follows from a straightforward application of Szemerédi's (Weak Hypergraph) Regularity Lemma.
	Here we give an alternative (possibly simpler) argument.
	More precisely, we derive \cref{lem:blow-up-matching} by iteratively applying the following result.
	In each step, we turn a tiling $\cB$ with `big' tiles $B$ into a tiling $\cL$ of `little' tiles $L$ covering additional $(\mu/8)^2n$ vertices.
	So we arrive at \cref{lem:blow-up-matching} after about $(8/\mu)^2$ steps.
	
	\begin{lemma}\label{lem:larger-matching}
		Let	$1/n \ll 1/m \ll  1/\ell \ll 1/s, \mu$, and let $P$ be an $s$-graph on $n$ vertices with $\delta_1(P) \geq \left(1-1/s+ \mu \right) \binom{n-1}{s-1}$.
		Set $B = K_{s}^{(s)}(m)$ and $L= K_{s}^{(s)}(\ell)$.
		Suppose that $P$ contains a $B$-tiling $\cB$ on $\lambda n$ vertices with $\lambda \leq 1-\mu/8$.
		Then $P$ contains an $L$-tiling on at least $(\lambda + 2^{-6}\mu^2 ) n$ vertices.
	\end{lemma}
	
	\begin{proof}
		For convenience, we assume that $n$ is divisible by $ms$.
		Let $\cU$ be a   partition of $V(P)$ into parts of size $m$, which is obtained by taking the $s$ parts of each copy of $B$ in~$\cB$ and the other parts in an arbitrary fashion.
		Set $\beta = \mu/8$.
		Let $R$ be an $s$-graph with vertex set~$\cU$ and an edge $X$ if $P$ contains at least $4\beta m^{s}$ $X$-partite edges.\footnote{So $R$ plays the role of a reduced graph in the context of a Regularity Lemma.}
		Let $r$ be the number of vertices of $R$, and observe that $r = n/m$.
		\begin{claim}\label{cla:reduced}
			We have $\delta_1(R) \geq \left( 1-1/s+\mu/4 \right) \binom{n/m-1}{s-1}$.
		\end{claim}
		\begin{proof}[Proof of \cref{cla:reduced}]
			Fix a vertex $x$ in $R$.
			There are at most $r^{s-2} m^{s} \leq (\mu/8) m^{s} \binom{r-1}{s-1}$ edges of $P$ that have one vertex in the part of $x$ and more than two vertices in some part.
			Put differently, and using $\delta_1(P) \geq \left(1-1/s+ \mu \right) \binom{n-1}{s-1}$, there are at least $m^s (1-1/s+(3/4)\mu) \binom{r-1}{s-1}$ edges in $P$, which are $\cU$-partite and have a vertex in the part of $x$.
			Every edge of $R$ incident to $x$ can host at most $m^{s}$ of these edges.
			Every non-edge of $R$ incident to $x$ can host at most $4\beta m^s$ of these edges.
			Thus $$m^{s} \deg_R(x) + 4\beta m^s \binom{r-1}{s-1}  \geq m^s (1-1/s+(3/4)\mu) \binom{r-1}{s-1}.$$
			Solving for $\deg(x)$ and recalling that $4\beta = \mu/2$ yields the desired bound.
		\end{proof}
		Hence there is a perfect matching $\cM$ in~$R$ by \cref{lem:matching}.
		For each edge $X$ of $\cM$, we use \cref{thm:erd64} to greedily find an $L$-tiling on at least $2\beta sm$ vertices in the $X$-partite subgraph of $P$ induced by the parts of $X$.
		(This is possible, since an $L$-tiling of order at most $2\beta sm$ intersects with at most $2\beta m^s$ edges whose vertices are covered by~$X$.)
		Denote the union of these `fresh' $L$-tilings by $\cL_{\text{fresh}}$.
		So $\cL_{\text{fresh}}$ covers at least $2\beta (1-\lambda)n \geq 2\beta^2 n$ vertices outside of $V(\cB)$.
		Moreover, $\cL_{\text{fresh}}$ covers the same number of vertices in each part of $\cU$.
		We then pick a maximal $L$-tiling $\cL_{\text{rec}} \subset P$ in $V(\cB) \setminus V(\cL_{\text{fresh}})$ to `recycle' what is left of $\cB$.
		So $\cL_{\text{rec}}$ leaves at most $\ell-1$ vertices uncovered in each of the parts of $\cB$.
		Since $\cB$ is a $B$-tiling on $\lambda n$ vertices with $L = K_{s}^{(s)}(m)$, there are $\lambda n/(sm)$ of these parts.
		Hence the tiling $\cL_{\text{rec}}$ covers all but $(\ell-1)\lambda n/(sm) \leq \beta^2 n$ vertices of $V(\cB) \setminus V(\cL_{\text{fresh}})$.
		It follows that $\cL_{\text{fresh}} \cup \cL_{\text{rec}}$ covers at least $2\beta^2 n + \lambda n - \beta^2 n \geq (\lambda + \beta^2) n$ vertices.
	\end{proof}
	
	We remark that the bounds on the order $n$ in the above argument are quite large (tower-type).
	A more economical approach is discussed in the conclusion of \cite{lang2023tiling} and implemented, for instance, in \cite[Lemma~B.3]{langsanhueza2024dirac}.

\end{document}